\documentclass[10pt]{article}

\usepackage[hmargin=35mm,vmargin=40mm]{geometry}

\usepackage{amsmath}              
\usepackage{amssymb}              
\usepackage{amsthm}               
\usepackage{enumerate}            
\usepackage{mathrsfs}             
\usepackage{graphicx}             
\usepackage{cancel}               
\usepackage[tight]{subfigure}     

\newtheorem{theorem}{Theorem}[section]
\newtheorem{lemma}[theorem]{Lemma}
\newtheorem{corollary}[theorem]{Corollary}
\newtheorem{algorithm}[theorem]{Algorithm}
\newtheorem{observation}[theorem]{Observation}
\newtheorem{conjecture}[theorem]{Conjecture}

\theoremstyle{definition}
\newtheorem*{defn}{Definition}
\newtheorem*{remark}{Remark}

\newcommand{\mvad}{1\nobreakdash-4}
\newcommand{\mvda}{4\nobreakdash-1}
\newcommand{\mvdd}{4\nobreakdash-4}
\newcommand{\mvbc}{2\nobreakdash-3}
\newcommand{\mvcb}{3\nobreakdash-2}
\newcommand{\base}[1]{\beta_1(#1)}
\newcommand{\bdry}{\partial}
\newcommand{\graphmfd}[3]{#1\ \cup/{#3}\ #2}
\newcommand{\meanor}[1]{\pi_{#1}}
\newcommand{\means}[1]{\sigma_{#1}}
\newcommand{\mobius}{M{\"o}bius}
\newcommand{\mfd}{\mathcal{M}}
\newcommand{\pg}[1]{\mathscr{P}(#1)}
\newcommand{\rpg}[1]{\mathscr{P}_1(#1)}
\newcommand{\R}{\mathbb{R}}
\newcommand{\regina}{\emph{Regina}}
\newcommand{\sfs}[2]{\mathrm{SFS}\left(#1: #2\right)}
\newcommand{\sfslong}{Seifert fibred space}
\newcommand{\sig}{\sigma}
\newcommand{\sss}{S^3}
\newcommand{\tri}{\mathcal{T}}

\newcommand{\homtwotable}[4]{
    \mbox{\tiny \renewcommand{\arraystretch}{1}
        $\! \left[ \begin{array}{@{\ }r@{\ }r@{\ }} #1 & #2 \\ #3 & #4
        \end{array} \right]$
    }
}

\begin{document}

\title{Simplification paths in the Pachner graphs \\
    of closed orientable 3-manifold triangulations}
\author{Benjamin A.~Burton}
\date{October 27, 2011} 

\maketitle

\begin{abstract}
    It is important to have effective methods for simplifying
    3-manifold triangulations without losing any topological information.
    In theory this is difficult: we might need to make a triangulation
    super-exponentially more complex before we can make it smaller than its
    original size.  Here we present experimental work
    that suggests the reality is far different:
    for an exhaustive census of 81\,800\,394 one-vertex triangulations
    that span 1\,901 distinct closed orientable 3-manifolds,
    we never need to add more than two extra tetrahedra, we never need
    more than a handful of Pachner moves (or bistellar flips),
    and the average number of Pachner moves decreases
    as the number of tetrahedra grows.
    If they generalise, these extremely surprising results
    would have significant implications for decision algorithms and the
    study of triangulations in 3-manifold topology.

    Key techniques include
    polynomial-time computable signatures that identify
    triangulations up to isomorphism,
    the isomorph-free generation of non-minimal triangulations,
    theoretical operations to reduce sequences of Pachner moves,
    and parallel algorithms
    for studying finite level sets in the infinite Pachner graph.



    \vspace{1cm}

    \centerline{\textbf{Journal version}} \medskip

    This is the journal version of \cite{burton11-pachner}, which
    was presented at the 27th Annual Symposium on Computational Geometry.
    This journal version contains significant new material.

    The study has been expanded from just 3-sphere triangulations to
    include all closed prime orientable 3-manifolds,
    it examines average-case as well as worst-case behaviour,
    and it also analyses moves
    that connect distinct minimal triangulations of the same 3-manifold.
    Section~\ref{s-analysis} contains several new algorithms, through which
    the loose upper bounds of \cite{burton11-pachner} have been replaced
    with smaller bounds, many of which are now tight.
    The analysis of pathological cases in Section~\ref{s-path} and
    the detailed specification of isomorphism signatures in the appendix
    are both new.
    The final discussion in Section~\ref{s-conc} is significantly richer,
    and raises new issues involving generic complexity.
\end{abstract}

\newpage

%
%

\section{Introduction} \label{s-intro}

Triangulations of 3-manifolds are ubiquitous in
computational knot theory
and low-dimen\-sion\-al topology.  They are easily obtained and offer
a natural setting for many important algorithms.

Computational topologists typically allow triangulations in which the
constituent tetrahedra may be ``bent'' or ``twisted'', and where
distinct edges or vertices of the same tetrahedron
may even be joined together.  Such triangulations (sometimes called
\emph{generalised triangulations})
can describe rich topological structures using remarkably few tetrahedra.
For example, the 3-dimensional sphere can be built from
just one tetrahedron, and more complex spaces such as non-trivial
{\sfslong}s can be built from as few as two \cite{matveev90-complexity}.

An important class of triangulations is the
\emph{one-vertex triangulations}, in which all vertices of all tetrahedra
are identified together as a single point.  These are simple to obtain
\cite{jaco03-0-efficiency,matveev03-algms},
and they are often easier to deal with both theoretically and computationally
\cite{burton10-dd,jaco02-algorithms-essential,matveev03-algms}.

Keeping the number of tetrahedra small is crucial in computational
topology, since many important algorithms are exponential (or even
super-exponential) in the number of tetrahedra
\cite{burton10-complexity,burton10-dd}.
To this end, topologists have developed a rich suite of local
moves that allow us to change a triangulation
without losing any topological information
\cite{burton04-facegraphs,matveev98-recognition}.
The ultimate aim is to \emph{simplify} the triangulation, i.e.,
reduce the number of tetrahedra, although the triangulation might
(temporarily) need to become more complex along the way.

The most basic of these moves are the four \emph{Pachner moves}
(also known as \emph{bistellar moves}).  These include the
{\mvcb} move (which reduces the number of tetrahedra but preserves the
number of vertices), the {\mvda} move (which reduces both numbers), and also
their inverses, the {\mvbc} and {\mvad} moves.
It is known that any two triangulations of the same closed
3-manifold are related by a sequence of Pachner moves
\cite{pachner91-moves}.  Moreover, if both are one-vertex
triangulations then in most cases
we can relate them using {\mvbc} and {\mvcb} moves alone \cite{matveev03-algms}.

However, little is known about how \emph{difficult} it is to
simplify a triangulation, or to convert one triangulation into another,
using Pachner moves.  In a series of papers,
Mijatovi{\'c} develops upper bounds on the number of moves required for
various classes of 3-manifolds \cite{mijatovic03-simplifying,mijatovic04-sfs,
mijatovic05-knot,mijatovic05-haken}.
All of these bounds are super-exponential in the number of tetrahedra,
and some even involve exponential towers of exponential functions.
For simplifying one-vertex triangulations using only {\mvbc} and {\mvcb} moves,
no explicit bounds are known at all.

Simplification is tightly linked to the important
\emph{recognition problem}, where we are given an input
triangulation $\tri$ and a target 3-manifold $\mfd$, and asked whether
$\tri$ triangulates $\mfd$.
The recognition problem is decidable but extremely difficult.  A general
algorithm comes as a consequence of Perelman's celebrated proof of the
geometrisation conjecture \cite{kleiner08-perelman},
but due to its intricate and
multi-faceted nature, the algorithm remains computationally intractable
and no explicit bound on its running time is known.

Some special cases of the recognition problem are more approachable.
A notable case is \emph{3-sphere recognition}
(where $\mfd=\sss$): this plays an key role in other important topological
algorithms such as connected sum decomposition
\cite{jaco03-0-efficiency,jaco95-algorithms-decomposition} and
unknot recognition \cite{hara05-unknotting},
and is also important for computational \emph{4-manifold} topology.
The original 3-sphere recognition algorithm of
Rubinstein \cite{rubinstein95-3sphere} has been improved significantly over time
\cite{burton10-dd,burton10-quadoct,jaco03-0-efficiency,
thompson94-thinposition},
and although it remains worst-case exponential, it is now highly
effective in practice for moderate-sized problems \cite{burton10-quadoct}.

We can use Pachner moves to solve recognition problems in two ways:
\begin{itemize}
    \item For the classes of manifolds $\mfd$ studied by Mijatovi{\'c}
    \cite{mijatovic03-simplifying,mijatovic04-sfs,
    mijatovic05-knot,mijatovic05-haken},
    and in particular the 3-sphere \cite{mijatovic03-simplifying},
    Pachner moves give a \emph{direct} recognition algorithm:
    select a well-known ``canonical'' triangulation $\tri_c$ of $\mfd$,
    and try to convert the input triangulation
    $\tri$ into $\tri_c$ by testing every possible
    sequence of Pachner moves up to Mijatovi{\'c}'s upper bound.
    Return ``true'' if and only if a successful conversion was found.

    \item For all manifolds $\mfd$, Pachner moves also give a \emph{hybrid}
    recognition algorithm: begin with a fast and/or greedy
    procedure to simplify $\tri$ as far as possible within a limited
    number of moves.  If we reach a well-known canonical triangulation of $\mfd$
    then return ``true''; otherwise run the full recognition algorithm
    (such as Rubinstein's algorithm for the case $\mfd=\sss$)
    on our new, and hopefully simpler, triangulation.
\end{itemize}

Direct algorithms are, at present, completely infeasible:
Mijatovi{\'c}'s bounds are super-exponential in the number of tetrahedra,
and the running times are super-exponential in Mijatovi{\'c}'s bounds.
Even for the trivial case of 3-sphere recognition with one tetrahedron,
a direct algorithm must test \emph{all possible} sequences of
$\sim 2.4 \times 10^{6027}$ Pachner moves.

The hybrid method, on the other hand, is found to be extremely effective
in practice.
Experience with 3-sphere recognition software \cite{burton04-regina}
suggests that when $\tri$ is indeed the 3-sphere,
the greedy simplification almost always gives a canonical triangulation,
which means that the slower Rubinstein method is almost never required.

In the context of Mijatovi{\'c}'s results, this effectiveness of the
hybrid method is
unexpected, and forms a key motivation for this paper.
More broadly, the aims of this paper are:
\begin{enumerate}[(i)]
    \item to measure how difficult it is \emph{in practice}
    to relate two triangulations of a 3-manifold using Pachner moves,
    or to simplify a 3-manifold triangulation to use fewer tetrahedra;

    \item to understand why greedy simplification techniques work so
    well in practice, despite the prohibitive theoretical bounds of
    Mijatovi{\'c};

    \item to investigate the possibility that Pachner moves could be
    used as the basis for a direct 3-sphere recognition algorithm
    that runs in \emph{sub-exponential} time.
\end{enumerate}

We restrict our attention to closed prime orientable 3-manifolds,
as well as the important case $\mfd=\sss$ (which is not prime).
We also restrict our attention to one-vertex triangulations with
{\mvbc} and {\mvcb} moves, which is the most relevant setting for computation.

Fundamentally this is an experimental paper, though the theoretical
underpinnings are interesting in their own right.
Based on an exhaustive census of almost 150~million
triangulations, including 81\,800\,394 one-vertex triangulations of
1\,901 distinct 3-manifolds,
the answers to the questions above appear to be:
\begin{enumerate}[(i)]
    \item we can relate and simplify one-vertex triangulations
    using remarkably few Pachner moves,
    and the average number of moves \emph{decreases} as the
    number of tetrahedra grows;

    \item both procedures require us to add \emph{at most two}
    extra tetrahedra, which explains why greedy simplification works so well;

    \item the number of moves required in the worst case to
    simplify a 3-sphere triangulation grows extremely slowly,
    to the point where sub-exponential time 3-sphere recognition may
    indeed be possible.
\end{enumerate}

These observations are extremely surprising, especially in light of
Mijatovi{\'c}'s bounds.  
For arbitrary manifolds, observation~(ii) does not generalise:
in Section~\ref{s-path} we construct larger triangulations of
graph manifolds, beyond the limits of our census, for which three
extra tetrahedra are required.
In the case of the 3-sphere, no such counterexamples are known.
If (iii) can be proven in general---yielding
a sub-exponential time 3-sphere recognition algorithm---this would be a
significant breakthrough in computational topology.

In Section~\ref{s-prelim} we outline preliminary concepts and
introduce \emph{Pachner graphs}, which are infinite graphs whose nodes
represent triangulations and whose arcs represent Pachner moves.
These graphs are the framework on which we build the rest of the paper.
We define \emph{simplification paths} through these graphs,
as well as the key quantities of \emph{length} and \emph{excess height}
that we seek to measure.

We follow in Section~\ref{s-tools} with two key tools for studying
Pachner graphs: an isomorph-free census of all closed 3-manifold
triangulations with $\leq 9$ tetrahedra (which gives us the nodes of the
graphs), and \emph{isomorphism signatures} of triangulations that can be
computed in polynomial time (which allow us to construct the arcs of the
graphs).  Here we also prove that the census grows at a super-exponential rate,
despite its strong topological constraints.

Section~\ref{s-analysis} introduces theoretical techniques for
``reducing'' paths through Pachner graphs, and describes parallel algorithms
that bound both the length and excess height of such paths.
These algorithms are designed to work within the severe time and memory
constraints imposed by the super-exponential census growth rate.
This section also presents the highly unexpected experimental results
outlined above.

In Section~\ref{s-path} we study pathological cases, including
the census triangulations that are most difficult to simplify, as well as the
graph manifold constructions mentioned above.
We finish in Section~\ref{s-conc} by exploring the wider implications of our
experimental results, in particular for the
worst-case and generic complexity analysis of
topological decision problems.

All code was written using the topological software package
{\regina} \cite{burton04-regina,regina}.

%
%

\section{Triangulations and the Pachner graph} \label{s-prelim}

A \emph{3-manifold triangulation of size $n$} is a collection of
$n$ tetrahedra whose $4n$ faces are affinely identified
(or ``glued together'') in $2n$ pairs so that the resulting
topological space is a closed 3-mani\-fold.\footnote{%
    It is sometimes useful to consider \emph{bounded} triangulations
    where some faces are left unidentified, or \emph{ideal}
    triangulations where the overall space only becomes a 3-manifold
    when we delete the vertices of each tetrahedron.
    Such triangulations do not concern us here.}
We are not interested in the shapes or sizes of tetrahedra (since these
do not affect the topology), but merely the combinatorics of how the faces are
glued together.
Throughout this paper, all triangulations and 3-manifolds are assumed to
be connected.

We do allow two faces of the same tetrahedron to be identified, and we also
note that distinct edges or vertices of the same tetrahedron might
become identified as a by-product of the face gluings.
A set of tetrahedron vertices that are identified together is collectively
referred to as a \emph{vertex of the triangulation}; we define an
\emph{edge} or \emph{face of the triangulation}
in a similar fashion.

\begin{figure}[htb]
    \centering
    \includegraphics{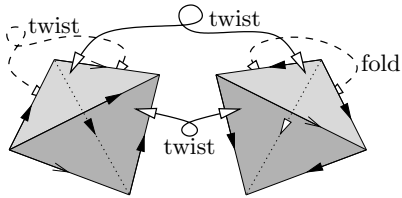}
    \caption{A 3-manifold triangulation of size $n=2$}
    \label{fig-rp3}
\end{figure}

Figure~\ref{fig-rp3} illustrates a 3-manifold triangulation
of size $n=2$.
Here the back two faces of the first tetrahedron are identified with a
twist, the front faces of the first tetrahedron are identified with the
front faces of the second using more twists, and the back faces of the
second tetrahedron are identified together by directly ``folding'' one
onto the other.
This is a \emph{one-vertex triangulation} since all eight tetrahedron
vertices become identified together.  The triangulation has three distinct
edges, indicated in the diagram by three distinct arrowheads.

For a given 3-manifold $\mfd$, a \emph{minimal triangulation} of $\mfd$
is a triangulation of $\mfd$ that uses the fewest possible tetrahedra.

Not every pairwise gluing of tetrahedron faces results in a
3-manifold triangulation.  Given $n$ tetrahedra whose faces are affinely
identified in pairs, we obtain a 3-manifold triangulation if and only if:
(i)~every vertex of the triangulation has a small regular neighbourhood
bounded by a sphere (not some higher-genus surface), and
(ii)~no edge of the triangulation is identified with itself in reverse.

The four \emph{Pachner moves} describe local
modifications to a triangulation.  These include:
\begin{itemize}
    \item the \emph{{\mvbc} move}, where we replace two distinct
    tetrahedra joined along a common face with three distinct tetrahedra
    joined along a common edge;
    \item the \emph{{\mvad} move}, where we replace a single
    tetrahedron with four distinct tetrahedra meeting at a common internal
    vertex;
    \item the \emph{{\mvcb}} and \emph{{\mvda} moves}, which are inverse to the
    {\mvbc} and {\mvad} moves.
\end{itemize}
These four moves are illustrated in Figure~\ref{fig-pachner}.
Essentially, the {\mvad} and {\mvda} moves retriangulate the interior of a
pyramid, and the {\mvbc} and {\mvcb} moves retriangulate the interior of a
bipyramid.
It is clear that Pachner moves do not change the topology of the
triangulation (i.e., the underlying 3-manifold remains the same).
Another important observation is that the {\mvbc} and {\mvcb} moves
do not change the number of vertices in the triangulation.

\begin{figure}[htb]
    \centering
    \subfigure[The 2-3 and 3-2 moves]{%
        \label{sub-pachner-23} \includegraphics[scale=0.45]{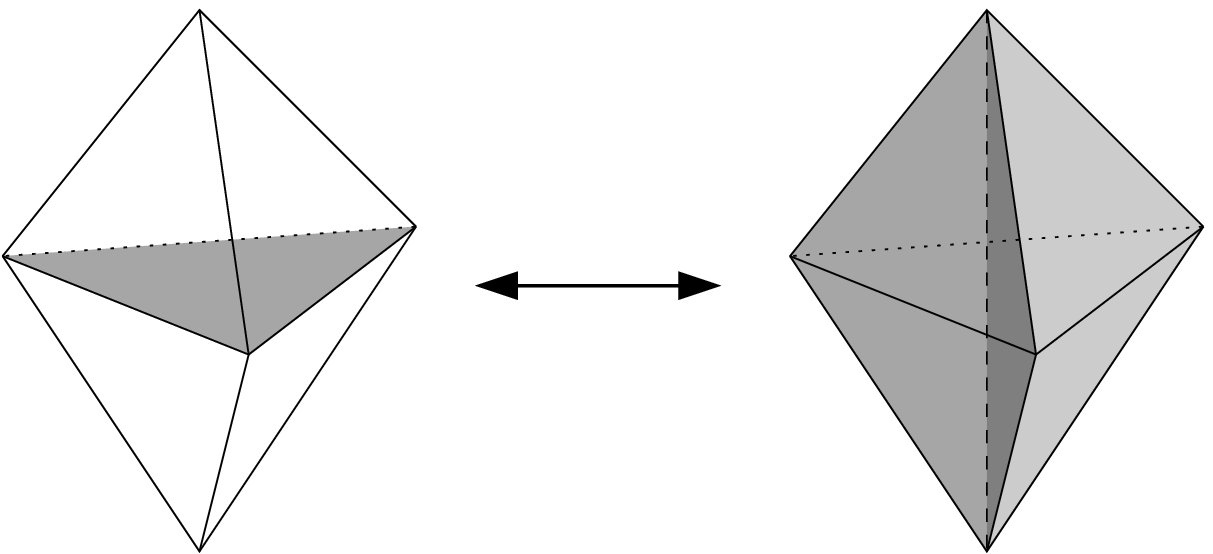}}
    \hspace{1.5cm}
    \subfigure[The 1-4 and 4-1 moves]{%
        \label{sub-pachner-14} \includegraphics[scale=0.45]{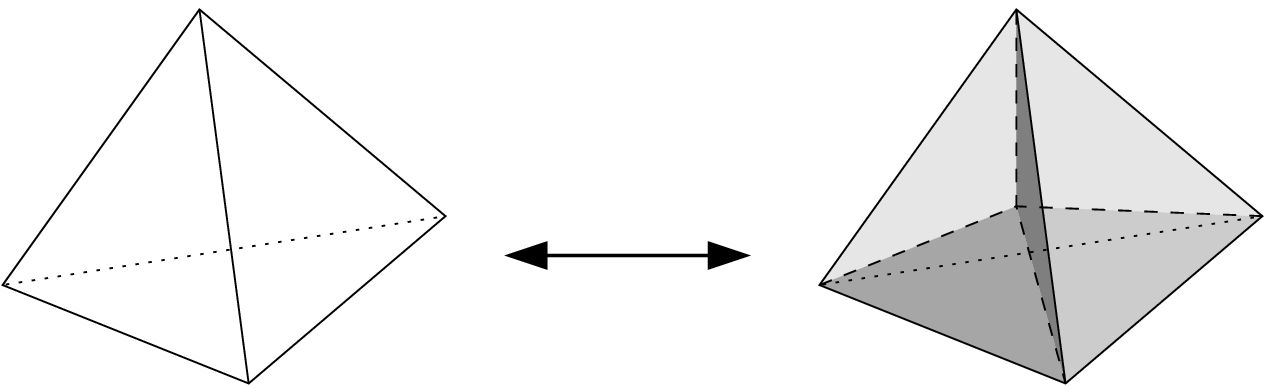}}
    \caption{The four Pachner moves for a 3-manifold triangulation}
    \label{fig-pachner}
\end{figure}

Two triangulations are \emph{isomorphic} if they are identical up to a
relabelling of tetrahedra and a reordering of the four vertices of each
tetrahedron (that is, isomorphic in the usual combinatorial sense).
Up to isomorphism, there are finitely many distinct triangulations of
any given size.

Pachner originally showed that any two
triangulations of the same closed 3-manifold can be made isomorphic by
performing a sequence of Pachner moves \cite{pachner91-moves}.\footnote{%
    As Mijatovi{\'c} notes,
    Pachner's original result was proven only for true simplicial complexes,
    but it is easily extended to the more flexible definition of a
    triangulation that we use here \cite{mijatovic03-simplifying}.
    The key step is to remove irregularities by performing
    a second barycentric subdivision using Pachner moves.}
Matveev later strengthened this result to show that any two
\emph{one-vertex} triangulations of the same closed 3-manifold
with at least two tetrahedra can be made isomorphic through a sequence
of {\mvbc} and/or {\mvcb} moves \cite{matveev03-algms}.
The two-tetrahedron condition is required because it is impossible to
perform a {\mvbc} or {\mvcb} move upon a one-tetrahedron triangulation
(each move requires two or three distinct tetrahedra).

In this paper we introduce the \emph{Pachner graph}, which describes
\emph{how} distinct triangulations of a closed 3-manifold can be
related via Pachner moves.
We define this graph in terms of \emph{nodes} and \emph{arcs}, to avoid
confusion with the vertices and edges that appear in
3-manifold triangulations.

\begin{defn}[Pachner graph]
    Let $M$ be any closed 3-manifold.  The \emph{Pachner graph} of $M$,
    denoted $\pg{M}$, is an infinite graph constructed as follows.
    The nodes of $\pg{M}$ correspond to isomorphism classes of
    triangulations of $M$.  Two nodes of $\pg{M}$ are joined by an
    arc if and only if there is some Pachner move that converts
    one class of triangulations into the other.

    The \emph{restricted Pachner graph} of $M$, denoted $\rpg{M}$,
    is the subgraph of $\pg{M}$ defined by only those nodes
    corresponding to one-vertex triangulations.
    The nodes of $\pg{M}$ and $\rpg{M}$ are partitioned into finite
    \emph{levels} $1,2,3,\ldots$, where each level~$n$ contains the nodes
    corresponding to $n$-tetrahedron triangulations.
\end{defn}

\begin{figure}[htb]
    \centering
    \includegraphics{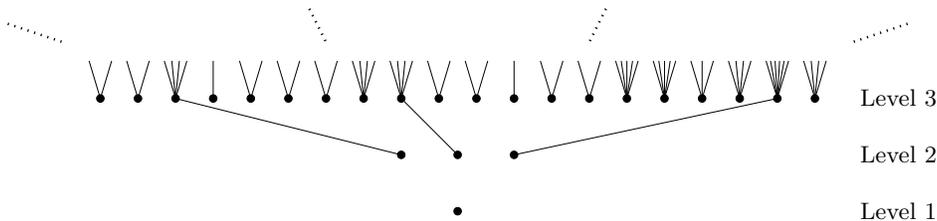}
    \caption{Levels 1--3 of the restricted Pachner graph of the 3-sphere}
    \label{fig-rpg-s3}
\end{figure}

It is clear that the arcs are well-defined (since Pachner moves are
preserved under isomorphism), and that arcs do not need to be
directed (since each {\mvbc} or {\mvad} move has a corresponding inverse
{\mvcb} or {\mvda} move).  In the full Pachner graph $\pg{M}$, each arc
runs from some level $i$ to a nearby level $i\pm1$ or $i\pm3$.
In the restricted Pachner graph
$\rpg{M}$, each arc must describe a {\mvbc} or {\mvcb} move,
and must run from some level $i$ to an adjacent level $i\pm1$.
Figure~\ref{fig-rpg-s3} shows the first few levels of the restricted Pachner
graph of the 3-sphere.

We can now reformulate the results of Pachner and Matveev as follows:

\begin{theorem}[Pachner, Matveev] \label{t-connected}
    The Pachner graph of any closed 3-manifold is connected.
    If we delete level~1,
    the restricted Pachner graph of any closed 3-manifold is also connected.
\end{theorem}

To simplify a triangulation we essentially follow a path through
$\pg{M}$ or $\rpg{M}$ from a higher level to a lower level,
motivating the following definitions:

\begin{defn}[Simplification path]
    A \emph{simplification path} is a directed path through either
    $\pg{M}$ or $\rpg{M}$ from a node at some level $i$ to
    a node at some lower level $<i$.
\end{defn}

\begin{defn}[Length and excess height]
    Let $p$ be any path through $\pg{M}$ or $\rpg{M}$ from level $i$ to
    level $j$.
    The \emph{length} of $p$ is the number of arcs it
    contains.  The \emph{excess height} of $p$ 
    is the smallest $h \geq 0$ for which
    the entire path stays in or below level $(\max\{i,j\} + h)$.
\end{defn}

\begin{figure}[htb]
    \centering
    \includegraphics{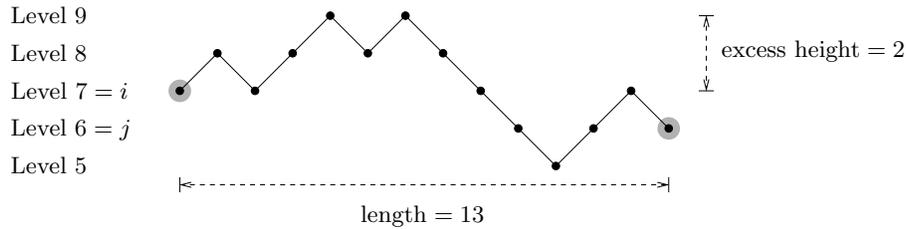}
    \caption{Measuring length and excess height}
    \label{fig-height}
\end{figure}

Figure~\ref{fig-height} illustrates a path of length~13 and excess height~2.
For simplification paths, the length and excess height measure how difficult
it is to simplify a
triangulation: the length measures the number of Pachner moves,
and the excess height measures the number of extra tetrahedra required.

To simplify a 3-sphere triangulation, the only known bounds on length
and excess height are due to Mijatovi{\'c} \cite{mijatovic03-simplifying}:

\begin{theorem}[Mijatovi{\'c}] \label{t-mij}
    Any triangulation of the 3-sphere can be converted into a
    two-tetrahedron triangulation using less than
    $6 \cdot 10^6 n^2 2^{2 \cdot 10^4 n^2}$ Pachner moves.
\end{theorem}

\begin{corollary}
    In the Pachner graph $\pg{\sss}$, from any node at level $n>2$
    there is a simplification path of length
    less than $6 \cdot 10^6 n^2 2^{2 \cdot 10^4 n^2}$ and excess height
    less than $3 \cdot 10^6 n^2 2^{2 \cdot 10^4 n^2}$.
\end{corollary}


Mijatovi{\'c} also proves bounds for other classes of 3-manifolds,
including Seifert fibred spaces \cite{mijatovic04-sfs},
fibre-free Haken manifolds \cite{mijatovic05-haken}, and
knot complements \cite{mijatovic05-knot}.  These all
involve towers of exponentials (where in some cases the height of the
tower grows with $n$), and the resulting bounds are far greater
than the 3-sphere bounds cited above.

In the \emph{restricted} Pachner graph,
where we only consider {\mvbc} and {\mvcb} moves,
no explicit bounds on length or excess height are known for any
3-manifolds at all.

%
%

\section{Key tools} \label{s-tools}

Experimental studies of Pachner graphs are difficult: the graphs
themselves are infinite, and even the finite level sets grow
super-exponentially in size (as we show in Theorem~\ref{t-numvert}).
By working with isomorphism classes of
triangulations, we keep the level sets considerably smaller than if
we had used labelled triangulations instead.
However, the trade-off is that both the nodes and the arcs of each
graph become more difficult to construct.

In this section we outline two key algorithmic tools for studying
Pachner graphs: a \emph{census of triangulations} (which enumerates the
nodes at each level), and polynomial-time computable
\emph{isomorphism signatures} (which allow us to construct the arcs).

\subsection{A census of triangulations} \label{s-tools-census}

To enumerate the nodes of Pachner graphs, we build an exhaustive census of all
3-manifold triangulations of size $n \leq 9$, with each triangulation
included precisely once up to isomorphism.  The total number of
triangulations in this census is $149\,676\,922$.  Following the focus of
this paper, we extract from these the one-vertex triangulations of
prime orientable 3-manifolds and the 3-sphere.  The resulting
$81\,800\,394$ triangulations represent $1\,901$
distinct 3-manifolds, and these triangulations
form the basis of our experiments.

A full breakdown of triangulations in the census appears in
Table~\ref{tab-census}.  For prime orientable 3-manifolds, we further
divide these triangulations into non-minimal (where we study simplification
paths) versus minimal (where we study how difficult it is to
join distinct minimal triangulations of the same 3-manifold).

\begin{table}[tb]
\centering
\small
\begin{tabular}{c||r|r||r|r|r|r}
Number &
\multicolumn{2}{c||}{No constraints} &
\multicolumn{4}{c}{One-vertex triangulations only} \\
\cline{2-7}
of &
\multicolumn{1}{c|}{All closed} &
\multicolumn{1}{c||}{3-spheres} &
\multicolumn{1}{c|}{All closed} &
\multicolumn{1}{c|}{3-spheres} &
\multicolumn{2}{c}{Prime and orientable} \\
tetrahedra &
\multicolumn{1}{c|}{3-manifolds} &
\multicolumn{1}{c||}{only} &
\multicolumn{1}{c|}{3-manifolds} &
\multicolumn{1}{c|}{only} &
\multicolumn{1}{c|}{Minimal} &
\multicolumn{1}{c}{Non-minimal} \\
\hline
1   &                4 &             2 &             3 &            1 &      2 &              \\
2   &               17 &             6 &            12 &            3 &      8 &              \\
3   &               81 &            32 &            63 &           20 &      7 &           31 \\
4   &              577 &           198 &           433 &          128 &     15 &          238 \\
5   &           5\,184 &        1\,903 &        3\,961 &       1\,297 &     40 &       2\,140 \\
6   &          57\,753 &       19\,935 &       43\,584 &      13\,660 &    115 &      22\,957 \\
7   &         722\,765 &      247\,644 &      538\,409 &     169\,077 &    309 &     272\,888 \\
8   &      9\,787\,509 &   3\,185\,275 &   7\,148\,483 &  2\,142\,197 &    945 &  3\,498\,286 \\
9   &    139\,103\,032 &  43\,461\,431 &  99\,450\,500 & 28\,691\,150 & 3\,031 & 46\,981\,849 \\
\hline
\textbf{Total} & 149\,676\,922 & 46\,916\,426 &
    107\,185\,448 & 31\,017\,533 & 4\,472 & 50\,778\,389 \\
\cline{5-7}
    & & & & \multicolumn{3}{c}{\textbf{81\,800\,394 of interest}} \\
\end{tabular}
\caption{Breakdown of different types of triangulations in the census}
\label{tab-census}
\end{table}

%
%

The algorithms behind this census are sophisticated:
\begin{itemize}
    \item \emph{Generating triangulations:}
    This is a combinatorial enumeration problem with severe
    topological constraints.
    If we simply enumerate all pairwise identifications of tetrahedron
    faces up to isomorphism, there are at least
    \[ \frac{[(4n-1)\times(4n-3)\times\cdots\times3\times1]\cdot6^{2n}}
        {n! \cdot 24^n} \quad \simeq \quad 2.35 \times 10^{16} \]
    possibilities for $n=9$.  However, the topological constraint that the
    triangulation must represent a 3-manifold
    cuts this number down to just
    $\simeq 1.39 \times 10^8$, as seen in Table~\ref{tab-census}.

    A key challenge therefore is to enforce this topological constraint as the
    census runs, and thus prune vast branches of the
    combinatorial search tree.
    Techniques for this include modified union-find and skip list
    algorithms for tracking partially-constructed edge and vertex links
    \cite{burton07-nor10,burton11-genus},
    and the analysis of 4-valent face pairing graphs
    \cite{burton04-facegraphs,burton07-nor10}.
    Some authors describe other techniques specific to minimal triangulations
    \cite{martelli06-or10,martelli01-or9,matveev98-recognition},
    but these are too specialised for the larger body of data
    that we require here.

    For the largest case $n=9$, the full enumeration of triangulations
    required $\sim 85$ days of CPU time as measured on a
    single 1.7~GHz IBM Power5 processor.
    In reality this was reduced to
    2--3 days of wall time using 32 CPUs in parallel.

    The paper \cite{burton11-genus} expands this
    census to $n=10$ (with $2\,046\,869\,999$ triangulations), but we do not
    use this data here because the resulting Pachner graphs are too large
    to process.  See Section~\ref{s-analysis} for further discussion on
    time and memory constraints.

    \item \emph{Identifying 3-spheres and orientable prime 3-manifolds:}
    Both 3-sphere recognition and prime decomposition are theoretically
    difficult problems.  Although the best known algorithms have exponential
    running times, recent advances have made them extremely fast
    for problems of our size.

    We employ algorithms based on the
    0-efficiency techniques of Jaco and Rubinstein
    \cite{jaco03-0-efficiency}, coupled with highly optimised algorithms
    for normal surface enumeration \cite{burton10-dd} and
    3-sphere recognition \cite{burton10-quadoct}.
    The total running time over all $\sim 150$ million triangulations
    was just 7.7~hours, which is negligible in comparison to the
    census enumeration.

    \item \emph{Identifying minimal and non-minimal triangulations:}
    For this we call upon a separate, specialised census of minimal
    3-manifold triangulations.  Minimal triangulations are extremely rare
    (as seen Table~\ref{tab-census}), which means that there are more
    opportunities for pruning the combinatorial search tree, and so
    specialised censuses of minimal triangulations are significantly
    faster to build.

    Censuses with \emph{at least one} minimal
    triangulation per prime orientable 3-manifold
    have been compiled by Matveev and others for
    $n \leq 12$ tetrahedra \cite{matveev05-or12},
    and censuses of \emph{all} minimal triangulations of such manifolds
    have been compiled by the author
    for $n \leq 11$ tetrahedra \cite{burton11-genus}.
    Since this study requires all minimal triangulations, we
    use \cite{burton11-genus} as our source.
\end{itemize}

It was noted
in the first point above that 3-manifold triangulations are extremely
rare amongst all pairwise identifications of tetrahedron faces.
Dunfield and Thurston \cite{dunfield06-random-covers} justify this
theoretically, proving
that as $n \to \infty$, the probability that a random identification of
tetrahedron faces yields a 3-manifold triangulation tends to zero.
Despite this rarity, we can prove that our census of 3-manifold
triangulations grows at a super-exponential rate:

\begin{theorem} \label{t-numvert}
    The number of distinct isomorphism classes of 3-manifold
    triangulations of size $n$ grows at an asymptotic rate of
    $\exp(\Theta(n\log n))$.
\end{theorem}

\begin{proof}
    An upper bound of $\exp(O(n\log n))$ is easy to obtain.
    If we count all possible pairwise gluings of tetrahedron faces, without
    regard for isomorphism classes or other constraints (such as the need for
    the triangulation to represent a closed 3-manifold), we obtain an
    upper bound of
    \[ \left[(4n-1)\times(4n-3)\times\cdots\times3\times1 \right]
        \cdot 6^{2n} < (4n)^{2n} \cdot 6^{2n} \in \exp(O(n\log n)). \]

    Proving a lower bound of $\exp(\Omega(n\log n))$ is more
    difficult---the main complication, as noted above, is that most pairwise
    identifications of tetrahedron faces do not yield a 3-manifold at all.
    We work around this by first counting
    \emph{2-manifold} triangulations (which are
    easier to obtain), and then giving a construction that ``fattens''
    these into 3-manifold triangulations without
    introducing any unwanted isomorphisms.

    To create a 2-manifold triangulation of size $2m$ (the size
    must always be even), we identify the $6m$ edges of $2m$
    distinct triangles in pairs.  Any such identification will always
    yield a closed 2-manifold; that is, nothing can ``go wrong'',
    in contrast to the three-dimensional case.%
        \footnote{Recall that things ``go wrong'' in three dimensions if
        a small neighbourhood of a vertex is not surrounded by a sphere,
        or if an edge is identified with itself in reverse.}

    There is, however, the issue of connectedness to deal with
    (recall from the beginning of Section~\ref{s-prelim} that all
    triangulations in this paper are assumed to be connected).  To ensure that
    a labelled 2-manifold triangulation is connected, we insist that for each
    $k=2,3,\ldots,2m$, the first edge of the triangle labelled $k$ is
    identified with
    some edge from one of the triangles labelled $1,2,\ldots,k-1$.  Of course
    many connected labelled 2-manifold triangulations do not have this
    property, but since we are proving a lower bound this does not matter.

    We can now place a lower bound on the number of \emph{labelled}
    2-manifold triangulations.  First we choose
    which edges to pair with the first edges from triangles
    $2,3,\ldots,2m$; from the property above we have
    $3 \times 4 \times \ldots \times 2m \times (2m+1) = \frac12 (2m+1)!$
    choices.
    We then pair off the remaining $2m+2$ edges, with
    $(2m+1) \times (2m-1) \times \ldots \times 3 \times 1 =
    (2m+1)!/2^m m!$ possibilities overall.
    Finally we note that each of the $3m$ pairs of edges can be identified
    using one of two possible orientations.
    The total number of labelled 2-manifold triangulations is therefore
    at least
    \[ \frac{(2m+1)!}{2} \cdot \frac{(2m+1)!}{2^m m!} \cdot 2^{3m}
        = \frac{(2m+1)! \cdot (2m+1)! \cdot 2^{2m}}{2 \cdot m!}. \]

    Each isomorphism class contains at most
    $(2m)! \cdot 6^{2m}$ labelled triangulations, and so the number of
    distinct \emph{isomorphism classes} of 2-manifold triangulations is
    bounded below by
    \begin{align*}
    \frac{(2m+1)! \cdot (2m+1)! \cdot 2^{2m}}
            {2 \cdot m! \cdot (2m)! \cdot 6^{2m}} &=
        \frac{(2m+1) \cdot (2m+1)!}{2 \cdot m! \cdot 3^{2m}} \\
        &> (2m+1) \times 2m \times \cdots \times (m+2) \times (m+1) \times
            \left(\tfrac{1}{9}\right)^m \\
        &> (m+1)^{m+1} \cdot \left(\tfrac{1}{9}\right)^m \\
        &\in \exp(\Omega(m\log m)).
    \end{align*}

    We fatten each 2-manifold triangulation into a 3-manifold triangulation
    as follows.  Let $F$ denote the closed 2-manifold described by the
    original triangulation.
    \begin{enumerate}
        \item Replace each triangle with a prism and glue the vertical
        faces of adjacent prisms together, as illustrated in
        Figure~\ref{sub-fatten-prisms}.
        This represents a \emph{bounded} 3-manifold, which is the
        product space $F \times I$.

        \item Cap each prism at both ends with a triangular pillow,
        as illustrated in Figure~\ref{sub-fatten-pillow}.
        The two faces of each pillow are glued to the top and bottom of
        the corresponding prism, effectively converting each prism into
        a solid torus.  This produces the \emph{closed} 3-manifold
        $F \times S^1$, and the complete construction is illustrated in
        Figure~\ref{sub-fatten-all}.

        \item Triangulate each pillow using two tetrahedra, which are joined
        along three internal faces surrounding an internal vertex.
        Triangulate each prism using $14$ tetrahedra, which again all
        meet at an internal vertex.
        Both triangulations are illustrated in Figure~\ref{sub-fatten-tri}.
    \end{enumerate}

    \begin{figure}[htb]
        \centering
        \begin{tabular}{c@{\qquad\qquad}c}
        \subfigure[Replacing triangles with prisms]{%
            \label{sub-fatten-prisms}%
            \includegraphics[scale=0.45]{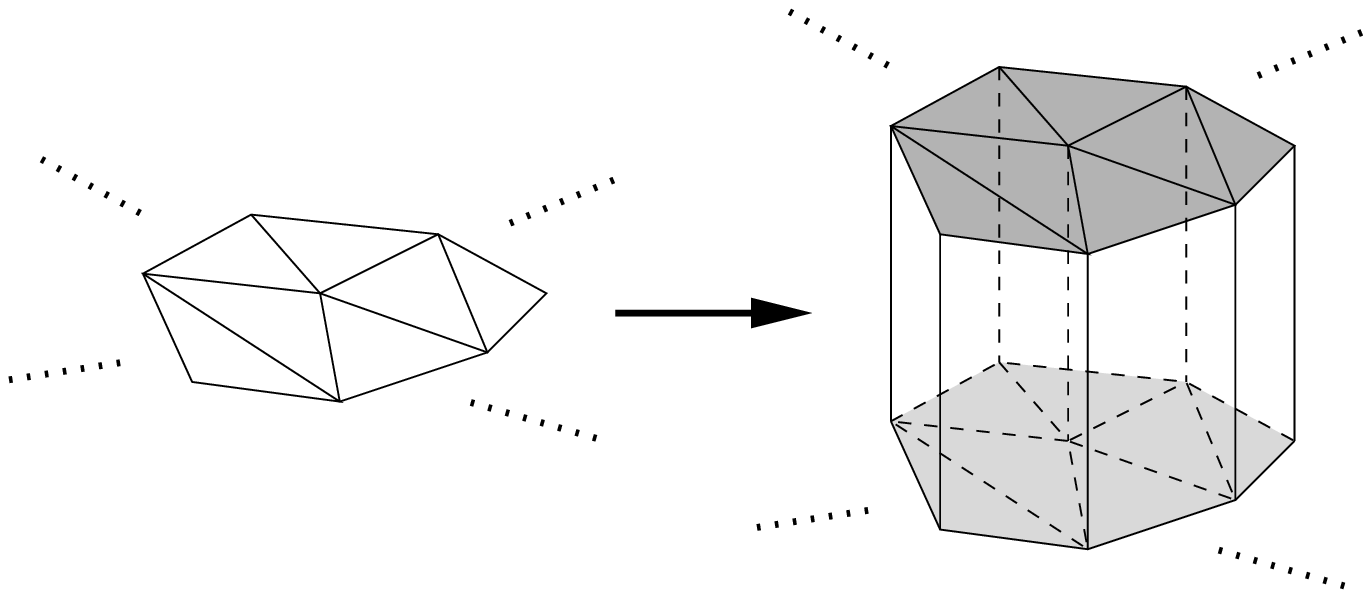}}%
        &
        \subfigure[Capping prisms with pillows]{%
            \hspace{2cm}%
            \label{sub-fatten-pillow}%
            \includegraphics[scale=0.45]{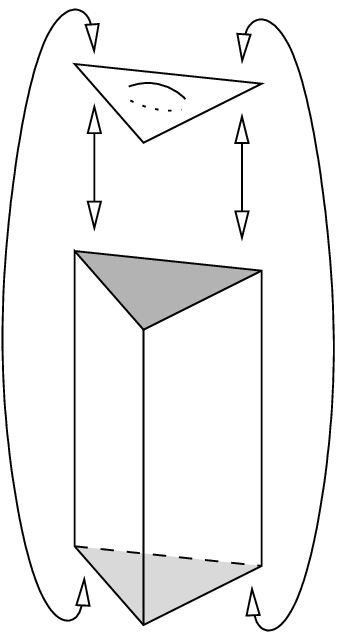}%
            \hspace{2cm}}
        \\
        \subfigure[The complete construction]{%
            \label{sub-fatten-all}%
            \includegraphics[scale=0.45]{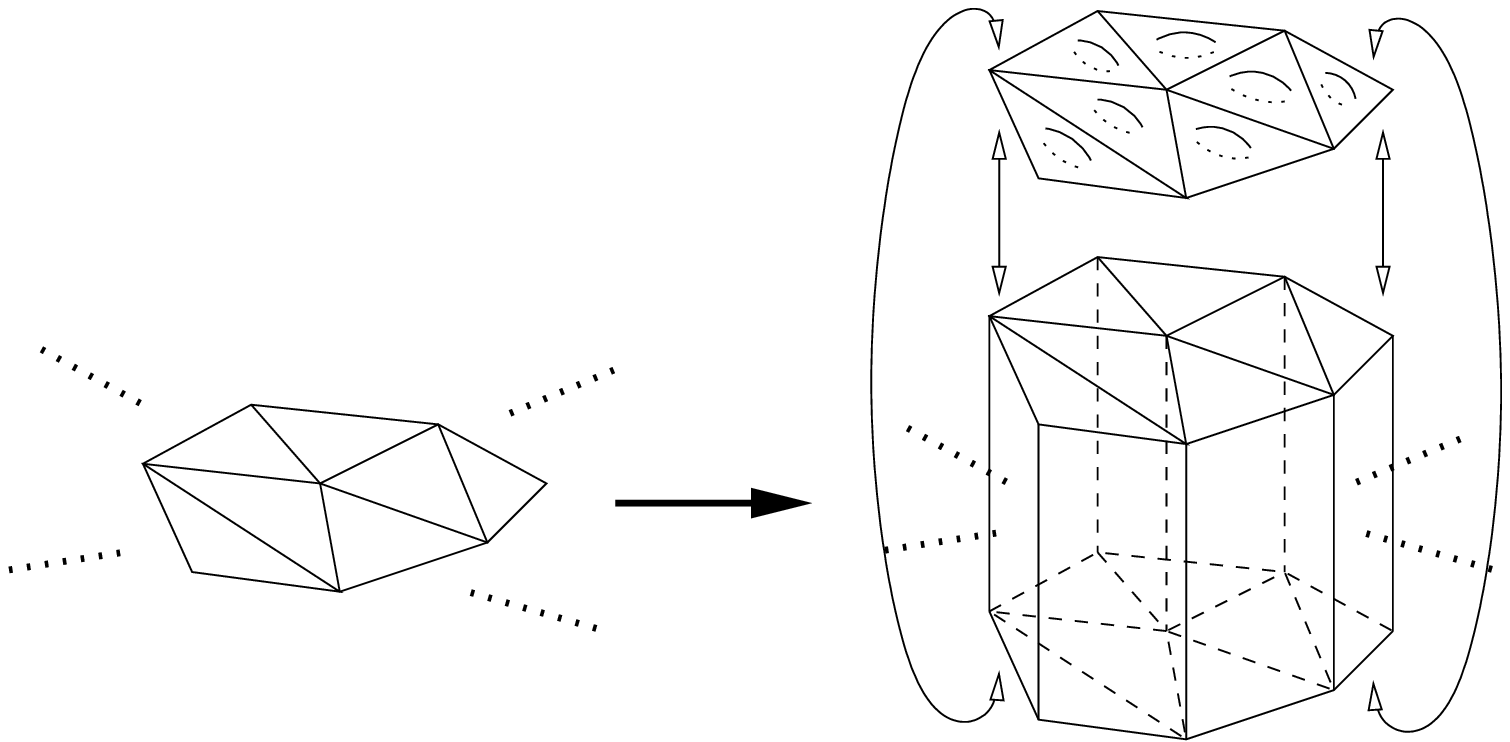}}
        &
        \subfigure[Triangulating prisms and pillows]{%
            \hspace{2cm}%
            \label{sub-fatten-tri}%
            \includegraphics[scale=0.45]{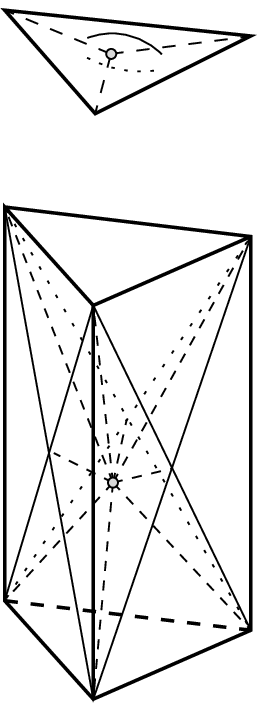}%
            \hspace{2cm}}
        \end{tabular}
        \caption{Fattening a 2-manifold triangulation into a 3-manifold
            triangulation}
        \label{fig-fatten}
    \end{figure}

    If the original 2-manifold triangulation uses $2m$ triangles, the
    resulting 3-manifold triangulation uses $n=32m$ tetrahedra.
    Moreover, if two 3-manifold triangulations obtained using this
    construction are isomorphic, the original 2-manifold triangulations
    must also be isomorphic.  The reason for this is as follows:
    \begin{itemize}
        \item Any isomorphism between two such 3-manifold triangulations
        must map triangular pillows to triangular pillows.  This is
        because the internal vertex of each triangular pillow meets only
        two tetrahedra, and no other vertices under our construction
        have this property.

        \item By ``flattening'' the triangular pillows into
        2-dimensional triangles, we thereby obtain an
        isomorphism between the underlying 2-manifold triangulations.
    \end{itemize}

    It follows that, for $n=32m$, we obtain a family of
    $\exp(\Omega(m\log m)) = \exp(\Omega(n\log n))$ pairwise
    non-isomorphic 3-manifold triangulations.

    This result is easily extended to $n \not\equiv 0 \bmod 32$.
    Let $V_n$ denote the number of distinct isomorphism classes of
    3-manifold triangulations of size $n$.
    \begin{itemize}
        \item Each triangulation of size $n$ has at least $n-1$
        distinct {\mvbc} moves available (since any face joining two
        distinct tetrahedra defines a {\mvbc} move, and there are at least
        $n-1$ such faces).

        \item On the other hand, each triangulation of size $n+1$
        has at most $6(n+1)$ distinct {\mvcb} moves available (since
        each {\mvcb}
        move is defined by an edge that meets three distinct tetrahedra,
        and the triangulation has at most $6(n+1)$ edges in total).
    \end{itemize}

    It follows that $V_{n+1} \geq V_n \cdot \frac{n-1}{6(n+1)} \geq V_n/18$
    for any $n > 1$.  This gives
    $V_{32m+k} \geq V_{32m} / 18^{31}$ for sufficiently large $m$ and
    all $0 \leq k < 32$, and so we obtain
    $V_n \in \exp(\Omega(n\log n))$ with no restrictions on $n$.
\end{proof}

\begin{remark}
    Of course, we expect that $V_{n+1} \gg V_n$ (and indeed we see this
    in the census).  The bounds
    that we use to show $V_{n+1} \geq V_n/18$ in the proof above are
    very loose, but they are sufficient for the asymptotic result that we seek.
\end{remark}

\subsection{Isomorphism signatures} \label{s-isosig}

To construct the arcs of a Pachner graph, we begin at a node---that is,
a 3-manifold triangulation $\tri$---and perform Pachner moves.
Our main difficulty is determining the endpoints of the arcs:
each Pachner move results in a new triangulation $\tri'$, and we must
determine which node of the graph represents $\tri'$.

A na\"ive approach might be to search through all nodes at the appropriate
level of the Pachner graph and test each triangulation
for isomorphism with $\tri'$.  However, this is infeasible:
although isomorphism testing is fast
(as we prove in Corollary~\ref{c-isotest}), the sheer number
of nodes at level $n$ of the graph is too large
(as shown by Theorem~\ref{t-numvert}).

What we need is a property of the triangulation $\tri'$ that is easy to
compute, and that uniquely defines the isomorphism class of $\tri'$.
This property can be used as the key in a data structure with fast
insertion and fast lookup (such as a hash table or a red-black tree),
and by computing this property we can quickly jump to
the relevant node of the graph.

Here we define such a property, which we call the \emph{isomorphism
signature} of a triangulation.
In Theorem~\ref{t-sig-unique} we show that isomorphism signatures
uniquely define isomorphism classes, and in
Theorem~\ref{t-sig-fast} we show that they are small to store
and fast to compute.

A \emph{labelling} of a triangulation of size $n$ involves:
(i)~numbering its tetrahedra from 0 to $n-1$ inclusive, and
(ii)~numbering the four vertices of each tetrahedron from 0 to 3
inclusive.\footnote{%
    We start numbering from 0 instead of 1 for consistency with the software
    implementation of isomorphism signatures, as detailed in the appendix.}
We also label the four faces of each tetrahedron from 0 to 3
inclusive so that face $i$ is opposite vertex $i$.
A key ingredient of isomorphism signatures is
\emph{canonical labellings}, which we define as follows.

\begin{defn}[Canonical labelling]
    Given a labelling of a triangulation of size $n$,
    let $A_{t,f}$ denote the tetrahedron glued to face
    $f$ of tetrahedron $t$ (so that $A_{t,f} \in \{0,\ldots,n-1\}$ for
    all $t=0,\ldots,n-1$ and $f=0,\ldots,3$).  The labelling is
    \emph{canonical} if, when we write out the sequence
    $A_{0,0},A_{0,1},A_{0,2},A_{0,3},\allowbreak A_{1,0},\ldots,A_{n-1,3}$,
    the following properties hold:
    \begin{enumerate}[(i)]
        \item For each $1 \leq i < j$,
        tetrahedron $i$ first appears before tetrahedron $j$
        first appears.
        \item For each $i \geq 1$, suppose tetrahedron $i$ first appears
        as the entry $A_{t,f}=i$.  Then the corresponding gluing
        uses the \emph{identity map}:
        face $f$ of tetrahedron $t$ is glued to face $f$ of tetrahedron $i$
        so that vertex $v$ of tetrahedron $t$ maps to vertex $v$ of
        tetrahedron $i$ for each $v \neq f$.
    \end{enumerate}
\end{defn}

As an example, consider the triangulation of size $n=3$ described by
Table~\ref{tab-gluings}.  This table lists the precise gluings of
tetrahedron faces.  For instance, the second cell in the bottom row
indicates that face~1 of tetrahedron~2 is glued to tetrahedron~1, in
such a way that
vertices $0,2,3$ of tetrahedron~2 map to vertices $3,1,2$ of tetrahedron~1
respectively.
This same gluing can be seen from the other direction by examining
the first cell in the middle row.

\begin{table}[htb]
    \newcommand{\gap}{\hspace{2ex}}
    \centering
    \small
    \begin{tabular}{l|c|c|c|c}
    &
    \multicolumn{1}{c|}{Face 0} &
    \multicolumn{1}{c|}{Face 1} &
    \multicolumn{1}{c|}{Face 2} &
    \multicolumn{1}{c}{Face 3} \\
    & Vertices 123 & Vertices 023 & Vertices 013 & Vertices 012 \\
    \hline
    {Tet.\ 0} & Tet.\ 0:\gap120 & Tet.\ 1:\gap023 &
                Tet.\ 2:\gap013 & Tet.\ 0:\gap312 \\
    {Tet.\ 1} & \framebox{Tet.\ 2:\gap230} & Tet.\ 0:\gap023 &
                Tet.\ 1:\gap012 & Tet.\ 1:\gap013 \\
    {Tet.\ 2} & Tet.\ 2:\gap012 & \framebox{Tet.\ 1:\gap312} &
                Tet.\ 0:\gap013 & Tet.\ 2:\gap123
    \end{tabular}
    \caption{The tetrahedron face gluings for an example 3-tetrahedron
        triangulation}
    \label{tab-gluings}
\end{table}

It is simple to see that the labelling for this triangulation is canonical.
The sequence $A_{0,0},\ldots,A_{n-1,3}$ is
$0,1,2,0,\allowbreak 2,0,1,1,\allowbreak 2,1,0,2$
(reading tetrahedron numbers from left to right and then top to bottom
in the table), and tetrahedron~1 first appears before
tetrahedron~2 as required.  Looking closer,
the first appearance of tetrahedron~1 is in the second cell of the top
row where vertices $0,2,3$ map to $0,2,3$, and the first appearance of
tetrahedron~2 is in the subsequent cell where vertices $0,1,3$ map to
$0,1,3$.  In both cases the gluings use the identity map.

\begin{lemma} \label{l-can-fast}
    For any triangulation $\tri$ of size $n$, there are precisely
    $24n$ canonical labellings of $\tri$, and these can be enumerated in
    $O(n^2\log n)$ time.
\end{lemma}

\begin{proof}
    %
    For $n=1$ the result is trivial, since all $24=4!$ possible labellings
    are canonical.  For $n>1$ we observe that, if we choose
    (i)~any one of the $n$ tetrahedra to label as tetrahedron~0, and
    (ii)~any one of the $24$ possible labellings of its four vertices,
    then there is one and only one way to extend these choices to a canonical
    labelling of $\tri$.

    To see this, we can walk through the list of faces
    $F_{0,0},F_{0,1},F_{0,2},F_{0,3},F_{1,0},\ldots,F_{n-1,3}$,
    where $F_{t,i}$ represents face $i$ of tetrahedron $t$.
    The first face amongst $F_{0,0},\ldots,F_{0,3}$ that is joined to an
    unlabelled tetrahedron must in fact be joined to tetrahedron~1 using
    the identity map.  This allows us to deduce tetrahedron~1 as well as
    the labels of its four vertices.

    We inductively extend the labelling in this manner: once we have
    labelled tetrahedra $0,\ldots,k$ and their corresponding vertices,
    the first face amongst $F_{0,0},\ldots,F_{k,3}$ that is joined to
    an unlabelled tetrahedron must give us tetrahedron $k+1$ along with
    the labels for its four vertices (again using the identity map).
    The resulting labelling is canonical, and all of the labels can be
    deduced in $O(n\log n)$ time using a single pass through the list
    $F_{0,0},\ldots,F_{n-1,3}$.  The $\log n$ factor is for
    manipulating tetrahedron labels, each of which requires $O(\log n)$ bits.

    It follows that there are precisely $24n$ canonical labellings
    of $\tri$, and that these can be enumerated in $O(n^2\log n)$ time
    using $24n$ iterations of the procedure described above.
\end{proof}

\begin{defn}[Isomorphism signature]
    For any triangulation $\tri$ of size $n$,
    enumerate all $24n$ canonical labellings of $\tri$,
    and for each canonical labelling encode the full set of
    face gluings as a sequence of bits.
    We define the \emph{isomorphism signature} to be the
    lexicographically smallest of these $24n$ bit sequences, and we
    denote this by $\sig(\tri)$.
\end{defn}

For the theoretical results in this section, it suffices to treat
each sequence as a na\"ive bitwise encoding of a full table of face gluings
(such as Table~\ref{tab-gluings}).
In practical software settings however, we use a more compact alphanumeric
representation with less redundancy, which we specify in full detail in the
appendix.

\begin{theorem} \label{t-sig-unique}
    Given two 3-manifold triangulations $\tri$ and $\tri'$,
    we have $\sig(\tri) = \sig(\tri')$ if and only if
    $\tri$ and $\tri'$ are isomorphic.
\end{theorem}

\begin{proof}
    It is clear that $\sig(\tri) = \sig(\tri')$ implies that
    $\tri$ and $\tri'$ are isomorphic, since both signatures encode the
    same gluing data.  Conversely, if $\tri$ and $\tri'$ are isomorphic
    then their $24n$ canonical labellings are the same (though they
    might be enumerated in a different order).  In particular, the
    lexicographically smallest canonical labellings will be identical;
    that is, $\sig(\tri)=\sig(\tri')$.
\end{proof}

\begin{theorem} \label{t-sig-fast}
    Given a 3-manifold triangulation $\tri$ of size $n$,
    the isomorphism signature $\sig(\tri)$ has $O(n\log n)$ size
    and can be generated in $O(n^2\log n)$ time.
\end{theorem}

\begin{proof}
    To encode a full set of face gluings, at worst we require a table
    of gluing data such as Table~\ref{tab-gluings}, with $4n$ cells each
    containing four integers.
    Because some of these integers require $O(\log n)$ bits
    (the tetrahedron labels), it follows that the total size of
    $\sig(\tri)$ is $O(n \log n)$.

    The algorithm to generate $\sig(\tri)$ is spelled out
    explicitly in its definition.  The $24n$ canonical labellings of
    $\tri$ can be enumerated in $O(n^2\log n)$ time (Lemma~\ref{l-can-fast}).
    Because a full set of face gluings has size $O(n\log n)$,
    we can encode the $24n$ bit sequences and select the
    lexicographically smallest in $O(n^2\log n)$ time,
    giving a time complexity of $O(n^2\log n)$ overall.
\end{proof}

This space complexity of $O(n\log n)$ is the best we can hope
for, since Theorem~\ref{t-numvert} shows that the number of distinct
isomorphism signatures for size $n$ triangulations grows like
$\exp(\Theta(n \log n))$.

It follows from Theorems~\ref{t-sig-unique} and~\ref{t-sig-fast}
that isomorphism signatures are ideal tools for constructing arcs in the
Pachner graph, as explained at the beginning of this section.
Moreover, the relevant definitions and results are easily extended
to bounded and ideal triangulations (as described in the appendix).
We finish with a simple but important consequence of our results:

\begin{corollary} \label{c-isotest}
    Given two 3-manifold triangulations $\tri$ and $\tri'$ each of size $n$,
    we can test whether $\tri$ and $\tri'$ are isomorphic in
    $O(n^2\log n)$ time.
\end{corollary}

%
%

\section{Analysing Pachner graphs} \label{s-analysis}

As discussed in the introduction, our focus
is on one-vertex triangulations of closed prime orientable 3-manifolds
and of the 3-sphere.  We therefore direct our attention to
the restricted Pachner graph $\rpg{\mfd}$ of each such 3-manifold $\mfd$.

\begin{defn}[Base level]
    For any closed 3-manifold $\mfd$, the \emph{base level}
    $\base{\mfd}$ is the lowest non-empty level of the
    restricted Pachner graph $\rpg{\mfd}$, excluding the isolated level~1.
    In other words, $\base{\mfd}$ is the
    smallest $n \geq 2$ for which there exists an
    $n$-tetrahedron, one-vertex triangulation of $\mfd$.
\end{defn}

\begin{lemma} \label{l-base}
    For any closed prime orientable 3-manifold $\mfd$ except for the
    lens spaces $L(4,1)$ and $L(5,2)$, the base level $\base{\mfd}$ is
    simply the size of a minimal triangulation of $\mfd$.
    If $\mfd=L(4,1)$ or $L(5,2)$, then $\base{\mfd}=3$.
    For the 3-sphere (which is not prime), $\base{\sss}=2$.
\end{lemma}

\begin{proof}
    From Table~\ref{tab-census},
    there are only three one-vertex triangulations of closed 3-manifolds;
    these represent $\sss$, $L(4,1)$ and $L(5,2)$.  The next smallest
    one-vertex triangulations of these manifolds in the census have
    sizes $\base{\sss}=2$, $\base{L(4,1)}=3$ and $\base{L(5,2)}=3$.

    For any other closed prime orientable 3-manifold $\mfd$,
    either $\mfd=\mathbb{R}P^3$, $\mfd=L(3,1)$, or
    every minimal triangulation of $\mfd$ has one vertex
    \cite{matveev90-complexity}.
    In the latter case, the result follows immediately from the
    definition of $\base{\mfd}$.  For both $\mathbb{R}P^3$ and $L(3,1)$
    at least one minimal triangulation has one vertex, and so the result
    holds for these cases also.
\end{proof}

In this section, we analyse all $81\,800\,394$ triangulations of
interest of size $n \leq 9$ in our census
to obtain computer proofs of the following results:

\begin{theorem}[Excess heights of simplification paths] \label{t-results-height}
    Let $\mfd$ be a closed prime orientable 3-manifold or the 3-sphere.
    From any node of $\rpg{\mfd}$ at level $n$ where
    $\base{\mfd} < n \leq 9$,
    there is a simplification path of excess height $\leq 2$.
\end{theorem}

\begin{theorem}[Lengths of simplification paths] \label{t-results-length}
    Let $\mfd$ be a closed prime orientable 3-manifold.
    From any node of $\rpg{\mfd}$ at level $n$ where
    $\base{\mfd} < n \leq 9$,
    there is a simplification path of length $\leq 17$.
    For the 3-sphere, the bound is stronger:
    from any node of $\rpg{\sss}$ at level $n$ where $\base{\sss} < n \leq 9$,
    there is some simplification path of length $\leq 9$.
\end{theorem}

\begin{theorem}[Joining minimal triangulations] \label{t-results-min}
    Let $\mfd$ be a closed prime orientable 3-manifold that is not
    the lens space $L(3,1)$, and for which $\base{\mfd} \leq 9$.
    Then any two nodes at level $\base{\mfd}$ of $\rpg{\mfd}$ are
    joined by a path of length $\leq 18$ and excess height $\leq 2$.

    For the special case $L(3,1)$, there are precisely two nodes
    at level $\base{L(3,1)}=2$, and these are
    joined by a path of length~$6$ and excess height~$3$.
\end{theorem}

These results are astonishing, given
Mijatovi{\'c}'s super-exponential bounds for the 3-sphere
\cite{mijatovic03-simplifying}
and tower-of-exponential bounds for other manifolds
\cite{mijatovic04-sfs,mijatovic05-knot,mijatovic05-haken}.
Theorems~\ref{t-results-height} and~\ref{t-results-min}
state that we can simplify triangulations or convert between minimal
triangulations using at most two extra tetrahedra, and
Theorems~\ref{t-results-length} and~\ref{t-results-min} state that we
require very few Pachner moves.
The 3-sphere results in particular have important algorithmic
implications, which we discuss further in Section~\ref{s-conc}.

The algorithms behind these computer proofs are specialised, and use
novel techniques to control the enormous time and space requirements.
We return to these algorithmic issues shortly.  In the meantime, some
further comments on Theorems~\ref{t-results-height}--\ref{t-results-min}:

\begin{itemize}
    \item We restrict our results to levels $n \leq 9$ because, even though
    we have census data for $n=10$ (see \cite{burton11-genus}),
    the sheer size of the census makes it infeasible to
    compute the necessary properties of the Pachner graphs.
    As a result, Theorems~\ref{t-results-height} and
    \ref{t-results-length} cover $747$ distinct 3-manifolds with
    $\base{\mfd} < 9$, and Theorem~\ref{t-results-min} covers $1900$
    distinct 3-manifolds with $\base{\mfd} \leq 9$.

    \item Theorem~\ref{t-results-length} bounds the \emph{worst-case} length
    of the shortest simplification path.  If we bound the
    \emph{average} length instead, the results are much smaller still:
    for the case $n=9$,
    this length is less than $1.89$ when averaged over all
    triangulations of $\sss$,
    and less than $1.91$ when averaged over all
    triangulations of closed prime orientable 3-manifolds.
    We present these average-case results in detail
    in Section~\ref{s-analysis-bfs}.

    \item For \emph{arbitrary} closed prime orientable 3-manifolds, the height
    results do not generalise.  In Section~\ref{s-path}, we construct
    (i)~a triangulation of a graph manifold of size $n=10$ for which
    every simplification path has excess height $\geq 3$,
    and (ii)~two minimal triangulations of a graph manifold
    with size $n=10$ where every path joining them has excess height $\geq 3$.
    For the case of the 3-sphere, no such counterexamples are known.
\end{itemize}

For the remainder of this section, we describe the algorithms behind
Theorems~\ref{t-results-height}--\ref{t-results-min},
and present the experimental results in detail.
Our algorithms are constrained by the following factors:
\begin{itemize}
    \item Their time and space complexities must be
    close to linear in the number of nodes that they examine,
    due to the sheer size of the census.

    \item They cannot loop through all nodes in $\rpg{\mfd}$, since
    the graph is infinite.  They cannot even loop through all nodes
    at any level $n \geq 11$, since there are too many to enumerate.

    \item They cannot follow arbitrary breadth-first or depth-first
    searches through $\rpg{\mfd}$, since the graph is infinite and
    can branch heavily in the upward direction.\footnote{%
        In general, a node at level $n$ can have up to $2n$
        distinct neighbours at level $(n+1)$.}
\end{itemize}

Because of these limiting factors, we cannot run through the census
and directly measure the shortest length or smallest excess height of any
simplification path from each node.
Instead we develop fast, localised algorithms that allow us to
bound these quantities from above.
To our delight, these upper bounds are extremely effective in practice.

A key optimisation in many of our algorithms comes from
Theorem~\ref{t-reduce}, which allows us to ``reduce'' sequences of
Pachner moves to use fewer intermediate tetrahedra.
We state and prove this theorem in Section~\ref{s-analysis-reduce}.
Following this, we present the individual algorithms:
\begin{itemize}
    \item In Section~\ref{s-analysis-height} we give a fast algorithm based on
    union-find that bounds the heights of simplification paths.
    We run this algorithm over the census of 3-sphere triangulations,
    giving a computer proof for the 3-sphere case of
    Theorem~\ref{t-results-height}.

    \item In Section~\ref{s-analysis-bfs} we describe a
    multiple-source breadth-first search algorithm that bounds lengths
    of simplification paths.  By running this over the census of
    triangulations of the 3-sphere and closed prime orientable manifolds,
    we prove Theorem~\ref{t-results-length} as well as
    the outstanding closed prime orientable case of
    Theorem~\ref{t-results-height}.

    \item In Section~\ref{s-analysis-min} we present algorithms that
    incorporate both breadth-first search and union-find techniques to
    study paths between nodes representing minimal triangulations.
    By running these algorithms over the census of minimal triangulations
    of closed prime orientable manifolds, we obtain a computer proof of
    Theorem~\ref{t-results-min}.

    \item We finish in Section~\ref{s-analysis-perf} with a
    discussion of how these algorithms can be parallelised effectively,
    along with explicit measurements of running time and memory use.
\end{itemize}


\subsection{Reducing sequences of moves} \label{s-analysis-reduce}

The key idea behind reducing sequences of Pachner moves is that,
in many cases, we can interchange pairs of consecutive {\mvbc} and {\mvcb} moves
to become consecutive {\mvcb} and {\mvbc} moves instead, without changing the
final triangulation.
In the Pachner graph, this replaces a path segment from levels
$n \to (n+1) \to n$ with a path segment from levels $n \to (n-1) \to n$:
the length and endpoints of the overall path remain the same, and the
excess height will either stay the same or decrease (depending on
the heights of other segments of the path).

This type of interchange is not always possible.  The purpose of
Theorem~\ref{t-reduce} is to identify the ``bad cases'' where
consecutive {\mvbc} and {\mvcb} moves cannot be interchanged, so that we can
explicitly test for them in algorithms.  In summary, we find that every
bad {\mvbc}~/~{\mvcb} pair can be described by one of three ``composite moves''.
These composite moves are local modifications to a triangulation,
much like the Pachner moves (though a little more complex),
and are defined as follows.

\begin{defn}[Octahedron flip]
    An \emph{octahedron flip}, also known as a \emph{{\mvdd} move},
    involves four distinct tetrahedra surrounding a common edge of
    degree four.  The move essentially rotates the configuration,
    replacing it with four new tetrahedra surrounding a new edge of
    degree four that points in a different direction,
    as illustrated in Figure~\ref{fig-flip-octahedron}.
\end{defn}

\begin{figure}[htb]
    \centering
    \includegraphics[scale=0.45]{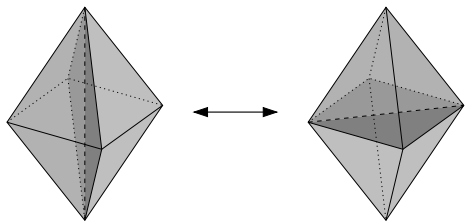}
    \caption{An octahedron flip}
    \label{fig-flip-octahedron}
\end{figure}

\begin{defn}[Pillow flip]
    A \emph{quadrilateral pillow}
    is formed from two distinct tetrahedra joined along
    two adjacent faces, as shown in Figure~\ref{sub-flip-pillow-pillow}.
    A \emph{pillow flip} is a move involving three tetrahedra:
    a quadrilateral pillow plus
    a third tetrahedron attached to one of its outer faces.
    The move essentially reflects the configuration, replacing it with
    a quadrilateral pillow plus a new tetrahedron attached to the
    opposite face instead.

    \begin{figure}[htb]
        \centering
        \subfigure[A quadrilateral pillow]{%
            \label{sub-flip-pillow-pillow}
            {\quad\includegraphics[scale=0.45]{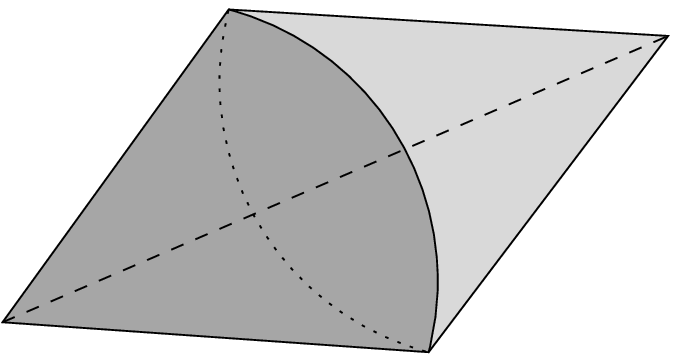}\quad}}
        \hspace{1.5cm}
        \subfigure[Performing a pillow flip]{%
            \label{sub-flip-pillow-move}
            \includegraphics[scale=0.9]{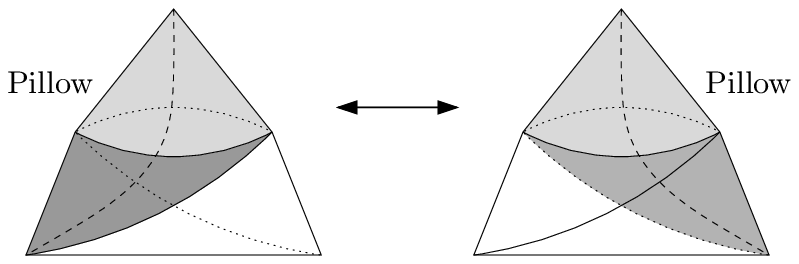}}
        \caption{A pillow flip}
        \label{fig-flip-pillow}
    \end{figure}

    This move is illustrated in Figure~\ref{sub-flip-pillow-move}, in
    which the
    pillows are shaded.  In the left diagram the pillow is at the front,
    and the third tetrahedron is attached to its lower rear face.
    In the right diagram the pillow moves to the rear,
    and the third tetrahedron is attached to its lower front face instead.
\end{defn}

\begin{defn}[Prism flip]
    There are two types of \emph{prism flip}, which we call types~A and~B.
    Both begin with three distinct tetrahedra joined to form a
    triangular prism, as illustrated in Figure~\ref{sub-flip-prism-prism}.
    For the type~A flip we also require that two rectangular faces of
    the prism are folded together, as shown in
    Figure~\ref{sub-flip-prism-fold}; for type~B we require that these
    same two rectangular faces be folded together with a $180^\circ$ twist
    instead.  In both cases, the
    prism flip involves rotating the entire configuration so that the
    two triangular ends of the prism are interchanged, as shown in
    Figure~\ref{sub-flip-prism-rotate}.
\end{defn}

\begin{figure}[htb]
    \centering
    \subfigure[A 3-tetrahedron triangular prism]{%
        \label{sub-flip-prism-prism}
        \includegraphics[scale=0.45]{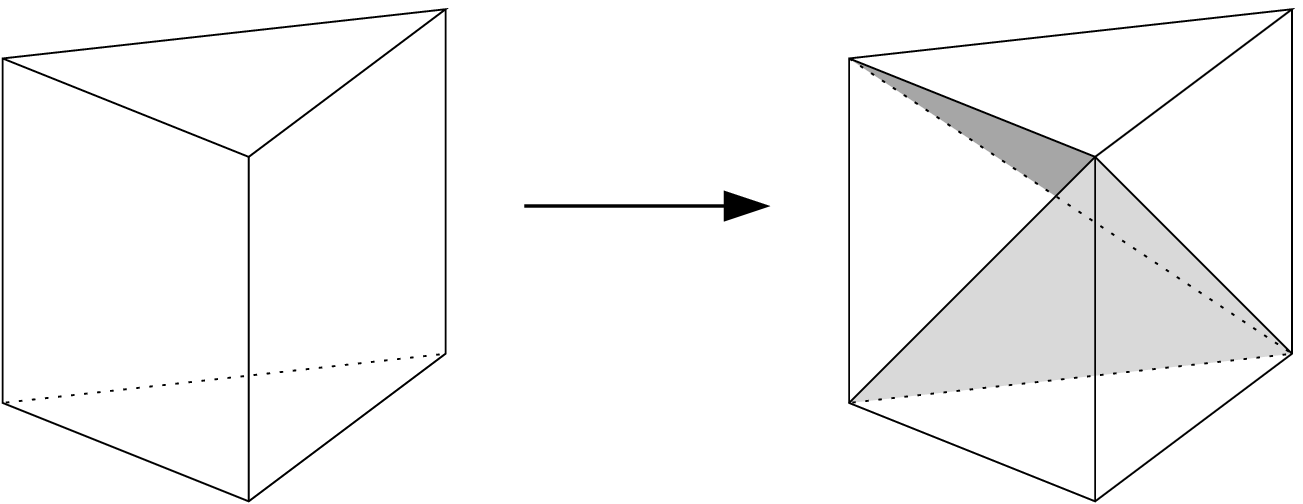}}
    \hspace{1.5cm}
    \subfigure[Folding two rectangular faces together]{%
        \label{sub-flip-prism-fold}
        {\quad\includegraphics[scale=0.45]{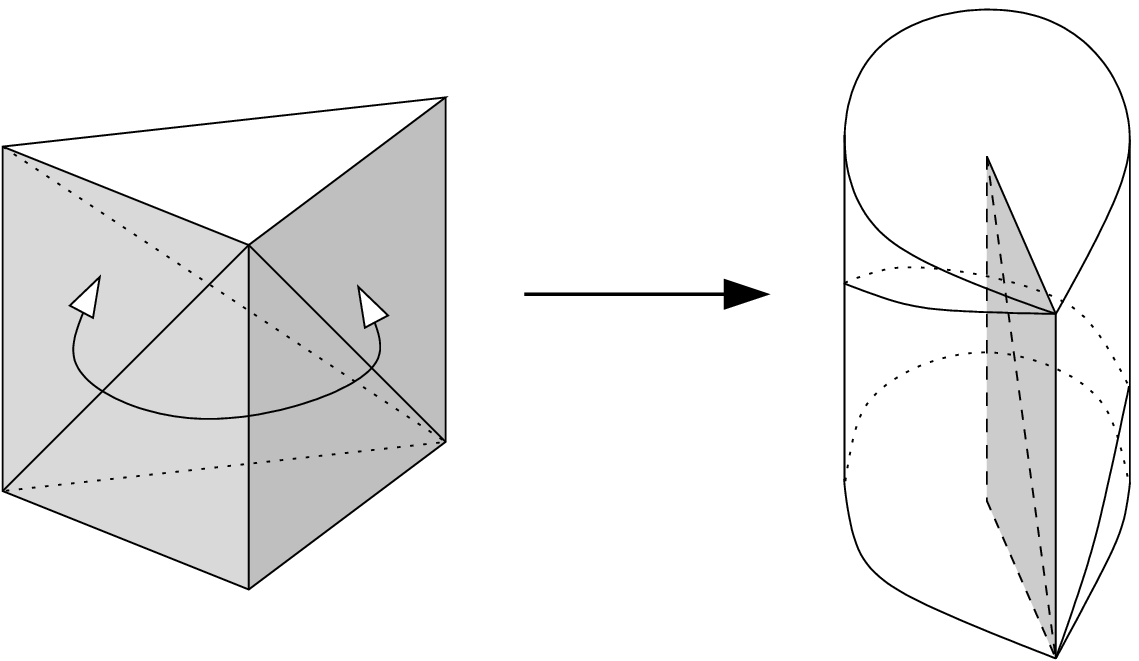}\quad}}
    \\
    \subfigure[Performing the move]{%
        \label{sub-flip-prism-rotate}
        \includegraphics[scale=0.9]{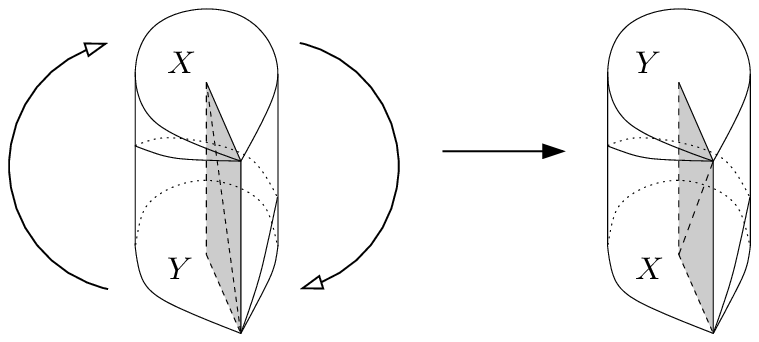}}
    \caption{A prism flip of type~A}
    \label{fig-flip-prism}
\end{figure}

\begin{remark}
    It is easy to search for all possible
    octahedron, pillow or prism flips on a given triangulation, since each
    occurs around an edge of degree four, two or five respectively.
    We can simply search for edges of the correct degree,
    test for the necessary configuration of tetrahedra,
    and perform the flip if possible.

    Note that there could be two different octahedron flips around the
    same degree four edge (corresponding to the two possible directions
    for the new edge that replaces it),
    and there could be four different pillow flips
    around the same degree two edge (corresponding to the four faces of
    the pillow to which the third tetrahedron might be attached).
    Around a degree five edge, there can only ever be one prism flip
    (if there is any at all).
\end{remark}

\begin{theorem} \label{t-reduce}
    Let $\tri$ and $\tri'$ be 3-manifold triangulations, where
    $\tri'$ can be obtained from $\tri$ by performing a {\mvbc} move
    followed by a {\mvcb} move.  Then one of the following cases holds:
    \begin{itemize}
        \item $\tri'$ is isomorphic to $\tri$;
        \item $\tri'$ can be obtained from $\tri$ by performing a
        {\mvcb} move followed by a {\mvbc} move instead (i.e., using
        one fewer tetrahedron at the intermediate stage instead of one more);
        \item $\tri'$ can be obtained from $\tri$ by performing either
        an octahedron flip, a pillow flip, or a prism flip.
    \end{itemize}
\end{theorem}

\begin{proof}
    Suppose that we obtain $\tri'$ from $\tri$ by:
    \begin{enumerate}[(i)]
        \item performing a {\mvbc} move on two adjacent tetrahedra
        $\Delta_1,\Delta_2$ of $\tri$, giving the intermediate
        triangulation $\mathcal{I}$;
        \item then performing a {\mvcb} move around the degree three edge $e$ of
        $\mathcal{I}$ to give $\tri'$.
    \end{enumerate}

    If $e$ is not an edge of the original triangulation $\tri$, then it
    must be created by the {\mvbc} move.  Therefore the
    subsequent {\mvcb} move is the inverse of the original {\mvbc} move,
    and the final triangulation $\tri'$ is isomorphic to $\tri$.

    Suppose then that $e$ does belong to the original triangulation $\tri$, but
    that $e$ is not an edge of either $\Delta_1$ or $\Delta_2$.  This implies
    that the {\mvbc} move and the {\mvcb} move occur in
    ``disjoint'' regions of
    the triangulation: none of the tetrahedra involved in the {\mvcb} move
    are also involved in the {\mvbc} move.
    Therefore the {\mvbc} and {\mvcb} moves
    can be applied in either order, and we can obtain
    $\tri'$ from $\tri$ by performing the {\mvcb}
    move followed by the {\mvbc} move instead.

    Finally, suppose that $e$ is an edge of $\Delta_1$ and/or $\Delta_2$
    in $\tri$.  Note that $e$ might not have degree three in $\tri$,
    and it might appear as multiple edges of $\Delta_1$ and/or $\Delta_2$.
    However, in the intermediate triangulation $\mathcal{I}$
    (after the initial {\mvbc} move), it must have degree three and it must
    belong to three distinct tetrahedra.

    Up to isomorphism, there are 13 possible ways that $e$ can appear
    in $\Delta_1$ and/or $\Delta_2$ subject to these constraints.
    These fall into five basic patterns, as illustrated in
    Figure~\ref{fig-across-cases}.  Cases 1, 2 and 4 require additional
    tetrahedra in order to meet the degree three requirement
    (these extra tetrahedra are also pictured).
    Cases 2, 3 and 5 each have several different subcases, according to the
    specific orientations of $e$ in each tetrahedron and the different
    ways in which tetrahedron faces can be glued together.

    \begin{figure}[htb]
        \centering
        \includegraphics[scale=0.9]{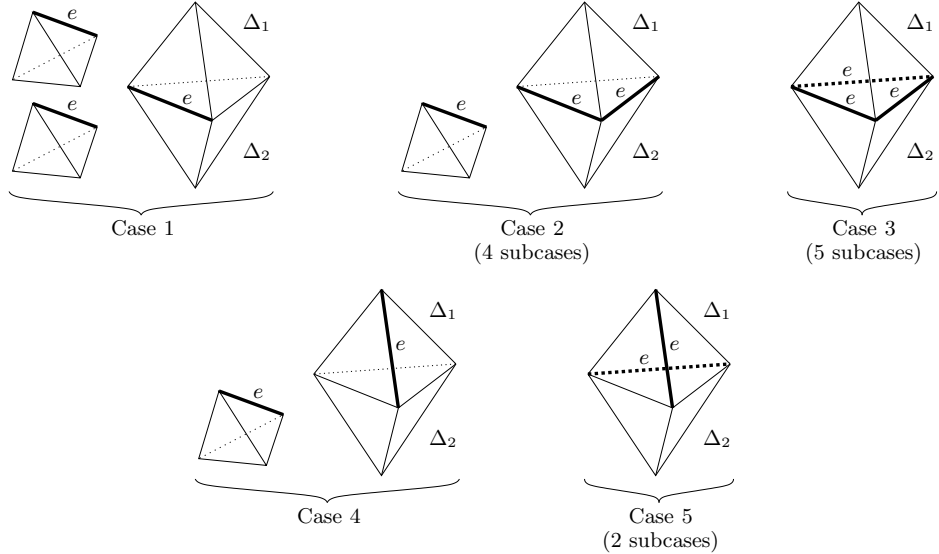}
        \caption{Possible locations for the edge $e$ in $\tri$}
        \label{fig-across-cases}
    \end{figure}

    Cases 1 and 4 are simple to analyse.
    Figure~\ref{fig-across-simple} shows the corresponding
    configurations in $\tri$ with the extra tetrahedra glued in:
    these are the initial configurations for an octahedron flip and a
    pillow flip respectively.  By following the details of
    the {\mvbc} move on $\Delta_1$ and $\Delta_2$ followed by
    the {\mvcb} move around $e$,
    we indeed find that these moves perform
    a single octahedron flip in case~1, and a single pillow flip in case~4.

    \begin{figure}[htb]
        \centering
        \includegraphics[scale=0.9]{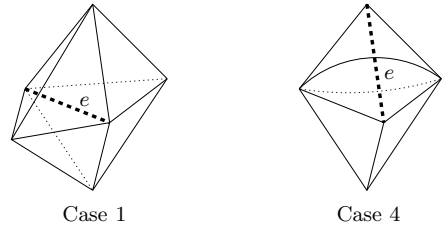}
        \caption{The full configurations in $\tri$ for cases 1 and 4}
        \label{fig-across-simple}
    \end{figure}

    For case~2, there are four possible configurations up to
    isomorphism.  These are shown in Figure~\ref{fig-across-prism},
    again with the extra tetrahedron glued in.
    These four subcases correspond to different possible orientations of $e$ in
    each tetrahedron (indicated in the diagram by the bold arrowheads) and
    different possible choices for which faces are glued together (indicated
    by the white arrows and the shading on the faces).

    \begin{figure}[htb]
        \centering
        \includegraphics[scale=0.9]{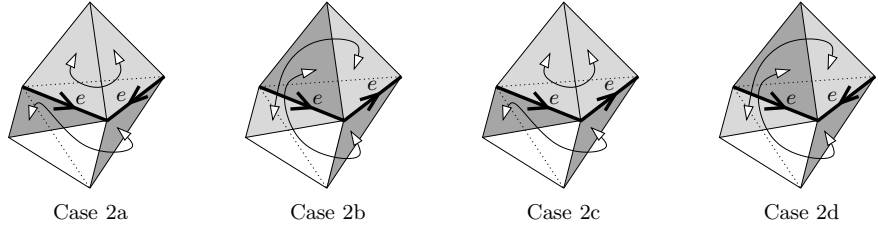}
        \caption{The four possible configurations for case 2}
        \label{fig-across-prism}
    \end{figure}

    Subcase~2a is precisely the initial configuration for a prism flip
    of type~A (the two shaded triangles on the left form one rectangular
    face of the prism, and the two shaded triangles on the right form another).
    Likewise, subcase~2b is the initial configuration for a prism flip
    of type~B (the different gluings describe a $180^\circ$ twist before
    folding these rectangles together).
    In each case, the {\mvbc} and {\mvcb} move together carry out the
    corresponding prism flip.

    Subcases~2c and 2d both produce non-orientable triangulations.
    More importantly, each contains a vertex $V$ with a non-orientable
    link---in other words, any small regular neighbourhood of $V$ is
    non-orientable.  For subcase~2c, this vertex is at the top of the
    diagram, and for subcase~2d it is the vertex at the centre of the diagram.
    No such vertex can appear in a 3-manifold triangulation, and so
    subcases~2c and 2d can never occur.

    In case~3, all eight faces of $\Delta_1$ and $\Delta_2$ are glued
    together in pairs.  Since the triangulation $\tri$ is connected,
    this means that $\tri$ contains only these two tetrahedra:
    if $n > 2$ then this case can never occur at all.  Even if $n=2$,
    the triangulation is not of interest to us, since
    $\tri$ will contain two distinct vertices and so will not
    feature in any restricted Pachner graph.

    Nevertheless, we summarise case~3 with $n=2$ for completeness.
    Here there are five subcases up to isomorphism, again depending
    on the possible orientations of $e$ and choices for how the
    the tetrahedron faces are glued together.
    Three of these subcases give non-orientable vertex links as before,
    and so cannot occur.
    Of the final two subcases, one gives $\tri$ as the unique two-vertex,
    two-tetrahedron triangulation of $\sss$, and the other gives
    the unique two-vertex, two-tetrahedron triangulation of the lens
    space $L(3,1)$.  In both subcases, the {\mvbc} and {\mvcb} move
    together produce a triangulation that is isomorphic to the original.%
    \footnote{This can be seen because the {\mvbc} and {\mvcb} moves
        together produce
        another two-vertex, two-tetrahedron triangulation of the same manifold,
        and the census shows that $\sss$ and $L(3,1)$ each have only one
        such triangulation up to isomorphism.  Of course one could simply
        follow through the moves by hand instead.}

    This leaves case~5.  Here there are two possible configurations in
    $\tri$ up to isomorphism, as illustrated in Figure~\ref{fig-across-deg3}
    (in both diagrams, the top left face at the front is glued to the
    upper face at the rear,
    and the top right face at the front is glued to the lower face at
    the rear).
    Subcase~5a gives an orientable triangulation (to be precise,
    a $(1,3,4)$ layered solid torus \cite{jaco06-layered}), and
    following through the {\mvbc} and {\mvcb} moves shows once again that the
    resulting triangulation is isomorphic to the original.
    Subcase~5b gives a non-orientable triangulation in which the
    top left edge is identified with itself in reverse,
    and so this subcase can never occur.
    \qedhere

    \begin{figure}[htb]
        \centering
        \includegraphics[scale=0.9]{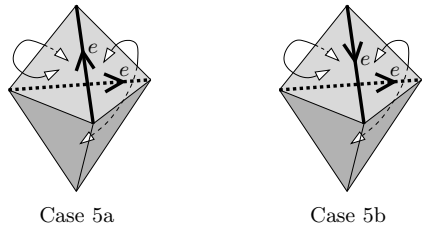}
        \caption{The two possible configurations for case 5}
        \label{fig-across-deg3}
    \end{figure}
\end{proof}


\subsection{Bounding excess heights} \label{s-analysis-height}

Our first algorithm computes bounds $H_n(\mfd)$ for a given
3-manifold $\mfd$ and a given level $n > \base{\mfd}$ so that, from every node
at level $n$ of the graph $\rpg{\mfd}$, there is a simplification path of
excess height $\leq H_n(\mfd)$.
By running this algorithm, we explicitly compute these bounds for the case
$\mfd=\sss$, for each $n$ in the range $2 < n \leq 9$.

The general idea is to begin with the set of all nodes at level $n$,
and repeatedly expand into higher levels using {\mvbc} moves only
until all of the original nodes are connected together.
If the bound $H_n(\mfd)$ is tight (as we find in practice for the 3-sphere),
this method of ``upward expansion'' saves significant time and memory by
enumerating only a fraction of the higher levels $\ell > n$
(each of which contains far more nodes than the initial level $n$).

\begin{algorithm}[Computing $H_n$] \label{a-height}
This algorithm runs by progressively building a
finite subgraph $G \subset \rpg{\mfd}$.
At all times we keep track of the number of distinct components of $G$
(denoted by $c$) and the maximum level of any node in $G$
(denoted by $\ell$).
\begin{enumerate}
    \item Initialise $G$ to all of level $n$ of $\rpg{\mfd}$.
    This means that $G$ has no arcs, the number of components $c$ is
    just the number of nodes at level $n$, and the maximum level is
    $\ell = n$.

    \item While $c > 1$, expand the graph as follows:
    \begin{enumerate}[(a)]
        \item Construct all arcs from nodes in $G$ at level $\ell$
        to (possibly new) nodes in $\rpg{\mfd}$ at level $\ell+1$.
        Insert these arcs and their endpoints into $G$.

        \item Update the number of components $c$, and increment $\ell$
        by one.
    \end{enumerate}

    \item Once we have $c=1$, output the final bound $H_n(\mfd) = \ell - n$
    and terminate.
\end{enumerate}
\end{algorithm}


In step~2(a) we construct arcs by performing {\mvbc} moves, and
in step~2(b) we use union-find to update the number
of components in small time complexity.

It is clear that Algorithm~\ref{a-height} is correct for any $n > \base{\mfd}$:
once we have $c=1$ the subgraph $G$ is connected, which means it contains a
path from any node at level $n$ to any other node at level $n$.
Because $n > \base{\mfd}$, Theorem~\ref{t-connected} shows that at least
one such node allows a {\mvcb} move, and so any node at level $n$ has a
simplification path of excess height $\leq \ell-n$.

However, it is not clear that Algorithm~\ref{a-height} terminates:
it might be that \emph{every} simplification path from some
node at level $n$ passes through nodes that we never construct
at higher levels $\ell > n$.
Happily it does terminate for the 3-sphere for all
$2 < n \leq 9$, giving an output of $H_n(\sss) = 2$ each time, and thereby
proving the 3-sphere case for Theorem~\ref{t-results-height}.
Table~\ref{tab-height} shows how
the number of components $c$ changes as the algorithm runs for each $n$.

\begin{table}[htb]
\[ \small \begin{array}{l|r|r|r|r|r|r|r}
    \mbox{Input level $n$} & 3 & 4 & 5 & 6 & 7 & 8 & 9 \\
    \hline
    \mbox{Value of $c$ when $\ell = n$} &
    20 & 128 & 1\,297 & 13\,660 & 169\,077 & 2\,142\,197 & 28\,691\,150 \\
    \mbox{Value of $c$ when $\ell = n+1$} &
    8 & 50 & 196 & 1\,074 & 7\,784 & 64\,528 & 557\,428 \\
    \mbox{Value of $c$ when $\ell = n+2$} &
    1 & 1 & 1 & 1 & 1 & 1 & 1 \\
    \hline
    \mbox{Final bound $H_n(\sss)$} &
    \mathbf{2} & \mathbf{2} & \mathbf{2} & \mathbf{2} & \mathbf{2} &
    \mathbf{2} & \mathbf{2}
\end{array} \]
\caption{Running Algorithm~\ref{a-height} on the 3-sphere for all
    levels $2 < n \leq 9$}
\label{tab-height}
\end{table}

\begin{lemma} \label{l-height-tight}
    If Algorithm~\ref{a-height} outputs $H_n(\mfd) \leq 2$, then
    this bound is tight.
    In other words,
    $H_n(\mfd) = \max_\nu \min_p \mathrm{excess~height}(p)$,
    where $\nu$ ranges over all nodes at level $n$ of $\rpg{\mfd}$,
    and where $p$ ranges over all simplification paths that begin at $\nu$.
\end{lemma}

\begin{proof}
    The result is trivial for $H_n(\mfd) \leq 1$.
    Suppose then that $H_n(\mfd)=2$ and this bound is not tight;
    that is, $\max_\nu \min_p \mathrm{excess~height}(p) \leq 1$.
    When Algorithm~\ref{a-height} processes level $\ell=n+1$,
    it effectively enumerates \emph{all} paths in $\rpg{\mfd}$
    that stay between levels $n$ and $n+1$, since every arc between these
    levels emanates from one of our starting nodes at level $n$.
    Therefore Algorithm~\ref{a-height} would terminate at
    level $\ell \leq n+1$ and output $H_n(\mfd) \leq 1$ instead,
    a contradiction.
\end{proof}

As a result, our bounds $H_n(\sss)=2$ for $2 < n \leq 9$ are all tight.
Note that Lemma~\ref{l-height-tight} does not work in the opposite direction:
it is possible to have $\max_\nu \min_p \mathrm{excess~height}(p) = 2$ but
$H_n(\mfd)>2$, since the optimal path might bounce around between levels
$n+1$ and $n+2$, using intermediate nodes that never appear in our
subgraph $G$.

It is straightforward to show that the space and time complexities of
Algorithm~\ref{a-height} are linear and log-linear respectively in the number
of nodes in $G$ (other small polynomial factors in $n$ and $\ell$
also appear).
Nevertheless, the memory requirements for $n=8$ were found to be extremely
large in practice ($\sim$30\,GB): by the time the algorithm terminated
at level $\ell=10$, we had visited $185\,697\,092$ nodes in total.
For $n=9$ the memory requirements were too large for the
algorithm to run (estimated at 400--500\,GB), and so a
\emph{two-phase} approach was necessary:

\begin{algorithm}[Two-phase algorithm for computing $H_n$]
    \label{a-height-hybrid}
    The following algorithm will either compute the same bound
    $H_n(\mfd)$ as Algorithm~\ref{a-height}, or will terminate with no result.
    Once again, we progressively build a finite subgraph $G \subset \rpg{\mfd}$
    and track the number of connected components $c$.

    \begin{enumerate}
        \item Initialise $G$ to all of level $n$ of $\rpg{\mfd}$, as before.

        \item Construct all arcs from nodes in $G$ at level $n$
        to (possibly new) nodes in $\rpg{\mfd}$ at level $n+1$.
        Insert these arcs and their endpoints into $G$, and update $c$
        accordingly.

        If $c=1$ then output $H_n(\mfd)=1$ and terminate.

        \item From each node $\nu$ at level $n+1$, try all possible
        octahedron flips, pillow flips and prism flips.
        Let $\nu'$ be the endpoint of such a flip (so $\nu'$ is also a
        node at level $n+1$).
        If $\nu' \in G$ then merge the components and
        decrement $c$ if necessary.  Otherwise do nothing
        (since $\nu'$ would never have been constructed in the previous
        algorithm).

        If $c=1$ then output $H_n(\mfd)=2$ and terminate;
        otherwise terminate with no result.
    \end{enumerate}
\end{algorithm}


In essence, we use the old Algorithm~\ref{a-height} for the transition
from level $n \to n+1$, but then we use Theorem~\ref{t-reduce} to
``simulate'' the transition from level $n+1 \to n+2$.  We only need to
consider octahedron, pillow and prism flips in step~3 because, by
Theorem~\ref{t-reduce}, if two nodes $\nu,\nu' \in G$ at level $n+1$
have arcs to some common node $\nu''$ at level $n+2$, then
either $\nu$ and $\nu'$ are related by such a flip, or else $\nu$ and $\nu'$
must already be connected in $G$.

It follows that, if this two-phase approach \emph{does}
output a result, it will always be the same result as Algorithm~\ref{a-height}.
The key advantage of this two-phase method is a much smaller memory
footprint (since it does not store any nodes at level $n+2$).
The main disadvantage is that it cannot move on to level $n+3$ if required,
and so if $H_n(\mfd)>2$ then it cannot output any result at all.

Of course by the time we reach $n=9$ for the 3-sphere, there are reasons to
suspect that $H_n(\sss)=2$ (following the pattern for $3 \leq n \leq 8$),
and so this
two-phase method seems a reasonable---and ultimately successful---approach.
For $n=9$ the memory consumption was $\sim$52\,GB, which was (just)
within the capabilities of the host machine.


\subsection{Bounding path lengths} \label{s-analysis-bfs}

Our next task is to compute bounds $L_n(\mfd)$
for a given 3-manifold $\mfd$ and a given level $n > \base{\mfd}$
so that, from every node at
level $n$ of $\rpg{\mfd}$, there is some simplification path of
length $\leq L_n(\mfd)$.  This time we compute
$L_n$ for all $n \leq 9$ and all closed prime orientable 3-manifolds
$\mfd$ in the census, as well as for the case $\mfd=\sss$.

It is infeasible to perform arbitrary breadth-first searches through
$\rpg{\mfd}$, and so we restrict such searches using two techniques:
\begin{enumerate}[(i)]
    \item we only search within levels $n$, $n+1$ and $n+2$, and we only
    explicitly store nodes in levels $n$ and $n+1$;
    \item we only visit level $n+2$ when absolutely necessary,
    as described by Theorem~\ref{t-reduce}.
\end{enumerate}

Of course (i) is only effective if there is a simplification path of
excess height $\leq 2$.  The results of the previous section
show that this is true for $\mfd=\sss$ and $n \leq 9$, and give us hope that
it holds for other manifolds also.  Although the \emph{shortest} simplification
paths might pass through levels $n+3$ or above (and so our bounds
$L_n(\mfd)$ might not be tight), the time and space savings
obtained by avoiding these higher levels (which grow at a super-exponential
rate) are enormous.

In the algorithm below,
we refer to \emph{steps} as the quantity minimised by the breadth-first
search.
Since each octahedron, pillow or prism flip involves two Pachner moves,
we use a modified breadth-first search for weighted graphs in which some
arcs are counted as one step and some arcs are counted as two.
Such modifications are standard, and we do not go into details here.

\begin{algorithm}[Computing $L_n$] \label{a-length}
    First identify the set $S$ of all nodes at level $n$ of
    $\rpg{\mfd}$ that have an arc running down to level $n-1$.
    Then conduct a multiple-source breadth-first search across levels
    $n$ and $n+1$, beginning with $S$ as the set of sources,
    where the steps in this breadth-first search are as follows:
    \begin{itemize}
        \item all arcs between levels $n$ and $n+1$ of $\rpg{\mfd}$,
        each of which is treated as one step;
        \item all octahedron, pillow and prism flips at level $n+1$
        of $\rpg{\mfd}$, each of which is treated as two steps.
    \end{itemize}
    If $s$ is the maximum number of steps required to reach
    any node in level $n$ from the initial source set $S$,
    then output the final bound $L_n(\mfd) = s+1$.
    If some nodes at level $n$ are never reached, then output
    $L_n(\mfd)=\infty$ instead.
\end{algorithm}

This time the algorithm always terminates, since levels $n$ and $n+1$ are
finite.  The algorithm is also correct, because each source node in $S$ has a
simplification path of length~1, and each step in the breadth-first search
corresponds to a single {\mvbc} or {\mvcb} move.
The space and time complexities are linear and log-linear
respectively in the number of nodes in levels $n$ and $n+1$ (again with
further small polynomial factors in $n$).

We can make some observations on the tightness of the bounds $L_n(\mfd)$:

\begin{lemma} \label{l-length-tight}
    The bound $L_n(\mfd)$ computed by Algorithm~\ref{a-length}
    is finite if and only if every node at level $n$ of $\rpg{\mfd}$
    has a simplification path of excess height $\leq 2$.

    If this bound is finite, then it is also tight
    if we restrict our attention to only simplification paths of excess height
    $\leq 2$.  That is, $L_n(\mfd) = \max_\nu \min_p \mathrm{length}(p)$,
    where $\nu$ ranges over all nodes at level $n$ of $\rpg{\mfd}$,
    and where $p$ ranges over all simplification paths of excess height
    $\leq 2$ that begin at $\nu$.
\end{lemma}

The proof follows directly from the structure of the breadth-first search
and from Theorem~\ref{t-reduce}, which shows that every step up into level
$n+2$ is either part of an octahedron, pillow or prism flip, or else can be
replaced with a step down into level~$n$ instead.

For the 3-sphere, Table~\ref{tab-length-s3} shows how the breadth-first
search progresses for each $n$ in the range $2 < n \leq 9$.
Each cell of this table counts only nodes at level $n$ of $\rpg{\sss}$,
not ``intermediate nodes'' at the higher level $n+1$.
For each $n$ the algorithm outputs a final bound of $L_n(\sss)=7$ or $9$,
proving the 3-sphere case of Theorem~\ref{t-results-length}.

\begin{table}[htb]
\[ \small \begin{array}{l|r|r|r|r|r|r|r}
    \mbox{Input level $n$} & 3 & 4 & 5 & 6 & 7 & 8 & 9 \\
    \hline
    \mbox{Nodes in $S$} &
        3 & 46 & 504 & 6\,975 & 91\,283 & 1\,300\,709 & 18\,361\,866 \\
    \mbox{Nodes 2 steps from $S$} &
        8 & 38 & 466 & 4\,315 & 54\,698 &    624\,144 &  8\,044\,998 \\
    \mbox{Nodes 4 steps from $S$} &
        7 & 43 & 309 & 2\,299 & 22\,634 &    213\,345 &  2\,255\,191 \\
    \mbox{Nodes 6 steps from $S$} &
        2 &  1 &  18 &     71 &     462 &        3988 &      29\,054 \\
    \mbox{Nodes 8 steps from $S$} &
          &    &     &        &         &          11 &           41 \\
    \hline
    \mbox{Total level $n$ nodes found} &
        20 & 128 & 1\,297 & 13\,660 & 169\,077 & 2\,142\,197 & 28\,691\,150 \\
    \mbox{Total level $n$ nodes missing} &
        0 & 0 & 0 & 0 & 0 & 0 & 0 \\
    \hline
    \mbox{Final bound $L_n$} &
    \mathbf{7} & \mathbf{7} & \mathbf{7} & \mathbf{7} & \mathbf{7} &
    \mathbf{9} & \mathbf{9}
\end{array} \]
\caption{Running Algorithm~\ref{a-length} over the 3-sphere}
\label{tab-length-s3}
\end{table}

\begin{observation}
    The bounds $L_n$ in Table~\ref{tab-length-s3} are tight even if we
    consider all simplification paths (not just those with
    excess height $\leq 2$).
\end{observation}

\begin{proof}
    For $n \leq 7$ this follows immediately from
    Lemma~\ref{l-length-tight}, since any simplification path of
    excess height $\geq 3$ must have length $\geq 7$.
    For $n=8$ or $9$, the bounds are likewise tight unless
    all $11$ or $41$ triangulations respectively that are eight steps from
    $S$ have a simplification path of length $\leq 8$ and excess height
    $\geq 3$.
    The only way to construct such a path is using three {\mvbc} moves
    followed by four {\mvcb} moves.  We can enumerate all such
    combinations of moves by computer, and we find in all 11 cases for $n=8$
    and in 40 of the 41 cases for $n=9$ that no such path exists.
\end{proof}

Table~\ref{tab-length-or} shows the corresponding results for arbitrary
closed prime orientable 3-mani\-folds.
Each cell in this table is a sum over all such manifolds $\mfd$ with
$\base{\mfd} < n$, except for the final row which lists the largest
bound $L_n(\mfd)$ amongst all such manifolds.
In every case we have $L_n(\mfd) \leq 17$, proving the closed prime
orientable case of Theorem~\ref{t-results-length}.

\begin{table}[htb]
\[ \small \begin{array}{l|r|r|r|r|r|r|r}
    \mbox{Input level $n$} & 3 & 4 & 5 & 6 & 7 & 8 & 9 \\
    \hline
    \mbox{Nodes in $S$} &
        8 & 80 & 931 & 11\,380 & 151\,278 & 2\,098\,537 & 30\,082\,708 \\
    \mbox{Nodes 2 steps from $S$} &
        2 & 73 & 682 &  7\,078 &  80\,998 &    991\,080 & 12\,579\,745 \\
    \mbox{Nodes 4 steps from $S$} &
        8 & 72 & 483 &  4\,163 &  38\,253 &    391\,516 &  4\,185\,195 \\
    \mbox{Nodes 6 steps from $S$} &
        1 & 13 &  44 &     324 &   2\,290 &     16\,714 &     131\,545 \\
    \mbox{Nodes 8 steps from $S$} &
          &    &     &      12 &       69 &         427 &       2\,591 \\
    \mbox{Nodes 10 steps from $S$} &
          &    &     &         &          &          12 &           51 \\
    \mbox{Nodes 12 steps from $S$} &
          &    &     &         &          &             &           13 \\
    \mbox{Nodes 14 steps from $S$} &
          &    &     &         &          &             &              \\
    \mbox{Nodes 16 steps from $S$} &
          &    &     &         &          &             &            1 \\
    \hline
    \mbox{Total level $n$ nodes found} &
        19 & 238 & 2\,140 & 22\,957 & 272\,888 & 3\,498\,286 & 46\,981\,849 \\
    \mbox{Total level $n$ nodes missing} &
        0 & 0 & 0 & 0 & 0 & 0 & 0 \\
    \hline
    \mbox{Maximum bound $L_n$} &
    \mathbf{7} & \mathbf{7} & \mathbf{7} & \mathbf{9} & \mathbf{9} &
    \mathbf{11} & \mathbf{17}
\end{array} \]
\caption{Running Algorithm~\ref{a-length} over arbitrary
    closed prime orientable 3-manifolds}
\label{tab-length-or}
\end{table}

Furthermore, by Lemma~\ref{l-length-tight},
an immediate consequence of these results is as follows.
For any closed prime orientable 3-manifold $\mfd$,
from any node at level $n$ of $\rpg{\mfd}$ where
$\base{\mfd} < n \leq 9$, there is a simplification path with
excess height $\leq 2$.
This proves the remaining closed prime
orientable case of Theorem~\ref{t-results-height}.

It is clear from Tables~\ref{tab-length-s3} and \ref{tab-length-or} that
most nodes at level $n$ are very close to the source set $S$, and only a
tiny fraction require all $L_n(\mfd)$ steps.  This prompts us to
consider \emph{average} lengths of simplification paths.
In particular, we consider the following two quantities:
\begin{itemize}
    \item $\means{n} = \mathrm{avg}_\nu \min_p \mathrm{length}(p)$,
    where $\nu$ ranges over all nodes at level $n$ of $\rpg{\sss}$,
    and where $p$ ranges over all simplification paths that begin at $\nu$;
    \item $\meanor{n} = \mathrm{avg}_\nu \min_p \mathrm{length}(p)$,
    where $\nu$ ranges over all nodes at level $n$ of all graphs $\rpg{\mfd}$
    for closed prime orientable 3-manifolds $\mfd$ with
    $\base{\mfd} < n$, and again
    $p$ ranges over all simplification paths that begin at $\nu$.
\end{itemize}
In other words, if we consider the shortest simplification path from
each node, then $\means{n}$ is the average path length over all size~$n$
triangulations of the 3-sphere, and $\meanor{n}$ is the average path length
over all size~$n$ triangulations of closed prime orientable 3-manifolds.

By aggregating the figures in Tables~\ref{tab-length-s3} and
\ref{tab-length-or}, we can place upper bounds on $\means{n}$ and
$\meanor{n}$ respectively (since each node that is $k$ steps from the
source set $S$ has a simplification path of length $\leq k+1$).
These upper bounds are listed in Table~\ref{tab-length-avg}
(all numbers are rounded up).

\begin{table}[htb]
\[ \small \begin{array}{l|r|r|r|r|r|r|r}
    \mbox{Input level $n$} & 3 & 4 & 5 & 6 & 7 & 8 & 9 \\
    \hline
    \mbox{Upper bound on $\means{n}$} &
        3.80 & 2.99 & 2.76 & 2.34 & 2.20 & 2.00 & 1.89 \\
    \mbox{(3-sphere triangulations)} & & & & & & & \\
    \hline
    \mbox{Upper bound on $\meanor{n}$} &
        3.22 & 3.16 & 2.67 & 2.44 & 2.21 & 2.05 & 1.91 \\
    \mbox{(closed prime orientable 3-manifolds)} & & & & & & & \\
\end{array} \]
\caption{Bounding the average lengths of simplification paths}
\label{tab-length-avg}
\end{table}

These averages are remarkably small, and more importantly, they appear
to \emph{decrease} with $n$:  although the worst-case paths become longer,
such pathological cases become very rare very quickly.
This has interesting implications for generic-case
complexity, which we return to in Section~\ref{s-conc}.
Figure~\ref{fig-steps} gives a graphical summary of these
worst-case and average-case bounds.

\begin{figure}[htb]
    \centering
    \includegraphics[scale=0.5]{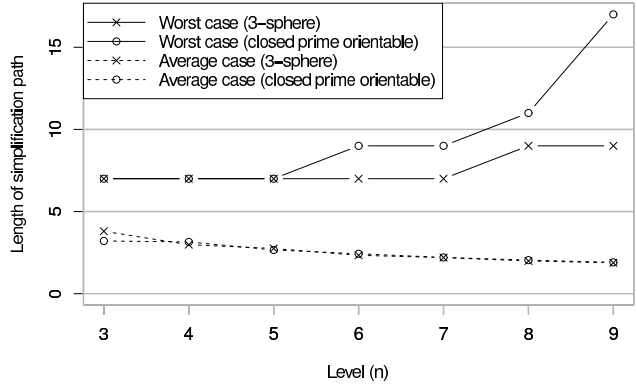}
    \caption{Summary of bounds on lengths of simplification paths}
    \label{fig-steps}
\end{figure}


\subsection{Connecting minimal triangulations} \label{s-analysis-min}

Our final algorithms examine paths that connect different nodes
at the base level $\base{\mfd}$ of $\rpg{\mfd}$.
Recall from Lemma~\ref{l-base} that, for most closed
prime orientable 3-manifolds $\mfd$, these are the nodes that correspond
to minimal triangulations of $\mfd$.

In particular, these algorithms compute bounds
$H_\mathrm{min}(\mfd)$ and $L_\mathrm{min}(\mfd)$
for a given 3-manifold $\mfd$ so that, between any two nodes at level
$\base{\mfd}$ of $\rpg{\mfd}$, there is a path of excess height
$\leq H_\mathrm{\min}(\mfd)$, and a (possibly different) path of length
$\leq L_\mathrm{\min}(\mfd)$.

Here the running time and memory use are less critical:
because the census contains so few minimal triangulations of closed
prime orientable 3-manifolds (just 4472 for $n \leq 9$),
we can afford to search exhaustively through entire levels of $\rpg{\mfd}$
as long as these levels do not grow too high above $\base{\mfd}$.

\begin{algorithm}[Computing $H_\mathrm{min}$] \label{a-height-min}
    As before, this algorithm runs by progressively building a
    finite subgraph $G \subset \rpg{\mfd}$.
    We track the number of distinct components of $G$
    (denoted by $c$) and the maximum level of any node in $G$
    (denoted by $\ell$).
    \begin{enumerate}
        \item Initialise $G$ to all of level $\base{\mfd}$ of $\rpg{\mfd}$,
        set the number of components $c$ to
        the number of nodes at level $\base{\mfd}$,
        and set the maximum level $\ell = \base{\mfd}$.

        \item While $c > 1$, expand $G$ using a multiple-source
        breadth-first search as follows:
        \begin{itemize}
            \item Maintain a queue of nodes to be processed.  Initially this
            queue should contain all nodes at level $\ell$ in $G$ (these are
            the multiple source nodes).
            \item To process a node $\nu$, follow all arcs from $\nu$ in
            $\rpg{\mfd}$ whose endpoints lie in levels $\leq \ell+1$,
            and add these arcs and their endpoints to $G$ if not
            already present.
            Each new node that is added to $G$ must also be
            pushed onto the queue for processing.
            Decrement $c$
            each time a new arc joins two disconnected components of $G$.

            \item Once the queue is empty (i.e., there are no more
            nodes to process), increment $\ell$.
        \end{itemize}

        \item Once we have $c=1$, output the final bound
        $H_\mathrm{min}(\mfd) = \ell - \base{\mfd}$ and terminate.
    \end{enumerate}
\end{algorithm}

As in Algorithm~\ref{a-height}, we enumerate arcs from a node $\nu$ by
performing {\mvbc} and {\mvcb} moves, and we track connected components of $G$
using union-find.

It is clear from the structure of the breadth-first search that,
at each stage of the algorithm, the subgraph $G$ contains precisely those
nodes that can be reached from level $\base{\mfd}$ of $\rpg{\mfd}$
using a path that never travels above level $\ell$.  We can therefore
conclude:

\begin{lemma} \label{l-height-min-tight}
    The bound $H_\mathrm{min}(\mfd)$ output by Algorithm~\ref{a-height-min}
    is tight.  That is,
    $H_\mathrm{min}(\mfd) =
    \max_{\nu_1,\nu_2} \min_p \mathrm{excess~height}(p)$,
    where $\nu_1$ and $\nu_2$ range over all nodes at level $\base{\mfd}$ of
    $\rpg{\mfd}$,
    and where $p$ ranges over all paths in $\rpg{\mfd}$ from
    $\nu_1$ to $\nu_2$.
\end{lemma}

Running this over all 1900 closed prime orientable 3-manifolds $\mfd$
with $\base{\mfd} \leq 9$, we find that $H_\mathrm{min}(\mfd) \leq 2$ in
every case but one.  The exception is one of the smallest cases in the
census: the lens space $L(3,1)$, with $\base{L(3,1)}=2$ and final bound
$H_\mathrm{min}(L(3,1))=3$.
Table~\ref{tab-height-min} shows a detailed breakdown of the results,
which together constitute a computer proof of the excess height results from
Theorem~\ref{t-results-min}.
Note that the manifolds in the table with $H_\mathrm{min}(\mfd)=0$
are those with only one node at level $\base{\mfd}$ of $\rpg{\mfd}$.

\begin{table}[htb]
\[ \small \begin{array}{l|r|r|r|r|r|r|r|r||r}
    \mbox{Base level $\base{\mfd}$}
        & 2 & 3 & 4 & 5 & 6 & 7 & 8 & 9 & \mbox{Total} \\
    \hline
    \mbox{Manifolds with $H_\mathrm{min}(\mfd)=0$} &
        6 &  7 & 13 & 23 & 47 & 106 & 235 & 575 & 1\,012 \\
    \mbox{Manifolds with $H_\mathrm{min}(\mfd)=1$} &
          &  1 &  1 &  7 & 20 &  52 & 156 & 455 & 692  \\
    \mbox{Manifolds with $H_\mathrm{min}(\mfd)=2$} &
          &  1 &    &  1 &  7 &  17 &  45 & 124 & 195 \\
    \mbox{Manifolds with $H_\mathrm{min}(\mfd)=3$} &
        1 &    &    &    &    &     &     &     & 1 \\
    \hline
    \mbox{Manifolds with $H_\mathrm{min}(\mfd) \geq 4$} &
        0 &  0 &  0 &  0 &  0 &   0 &   0 &   0 & 0 \\
\end{array} \]
\caption{Bounding excess heights of paths between nodes at level $\base{\mfd}$}
\label{tab-height-min}
\end{table}

To bound path \emph{lengths}, we adopt a similar strategy to that
for simplification paths: because excess heights are typically bounded
above by $H_\mathrm{min}(\mfd) \leq 2$, we perform breadth-first
searches that stay within levels $\base{\mfd}$, $\base{\mfd}+1$ and
$\base{\mfd}+2$.  Once again, we use octahedron, pillow and prism flips
to avoid explicitly stepping into level $\base{\mfd}+2$.
The full algorithm is as follows.

\begin{algorithm}[Computing $L_\mathrm{min}$] \label{a-length-min}
    For each node $\nu$ at level $\base{\mfd}$,
    run a single-source breadth-first search from $\nu$ across levels
    $\base{\mfd}$ and $\base{\mfd}+1$ of $\rpg{\mfd}$.
    The steps in this breadth-first search are as follows:
    \begin{itemize}
        \item all arcs between levels $\base{\mfd}$ and $\base{\mfd}+1$,
        each of which is treated as one step;
        \item all octahedron, pillow and prism flips at level $\base{\mfd}+1$,
        each of which is treated as two steps.
    \end{itemize}
    Let $d(\nu)$ denote the maximum number of steps required to reach
    any node at level $\base{\mfd}$ from the source $\nu$,
    or let $d(\nu)=\infty$ if some node at level
    $\base{\mfd}$ was never reached.
    After computing $d(\nu)$ for each source node $\nu$,
    output the final bound $L_\mathrm{\min}(\mfd) = \max_\nu d(\nu)$.
\end{algorithm}

This approach of running a separate breadth-first search from each node
$\nu$ at level $\base{\mfd}$ is perhaps wasteful, but there are so few
minimal triangulations in the census that its performance is
adequate nonetheless.

As with Algorithm~\ref{a-length} for simplification paths, we can make
some simple observations on the tightness of the bounds
$L_\mathrm{min}(\mfd)$.  The following result is an immediate
consequence of Theorem~\ref{t-reduce}, the breadth-first search
structure of Algorithm~\ref{a-length-min}, and the tightness
of the height bounds $H_\mathrm{min}(\mfd)$
as shown by Lemma~\ref{l-height-min-tight}.

\begin{lemma} \label{l-length-min-tight}
    The bound $L_\mathrm{min}(\mfd)$ that is
    computed by Algorithm~\ref{a-length-min}
    is finite if and only if $H_\mathrm{min}(\mfd) \leq 2$.
    Moreover, if the bound $L_\mathrm{\min}(\mfd)$ is finite,
    then it is also tight if we restrict our attention to
    paths of excess height $\leq 2$.  That is,
    $L_\mathrm{min}(\mfd) =
    \max_{\nu_1,\nu_2} \min_p \mathrm{length}(p)$,
    where $\nu_1$ and $\nu_2$ range over all nodes at level $\base{\mfd}$ of
    $\rpg{\mfd}$,
    and where $p$ ranges over all paths in $\rpg{\mfd}$ from
    $\nu_1$ to $\nu_2$ with excess height $\leq 2$.
\end{lemma}

Table~\ref{tab-length-min} shows the results of this algorithm when run
over all 1900 closed prime orientable 3-manifolds $\mfd$
with $\base{\mfd} \leq 9$.  There is just one case with
$L_\mathrm{min}(\mfd)=\infty$ (the case $\mfd=L(3,1)$ from before),
and for every other manifold we have $L_\mathrm{min}(\mfd) \leq 18$
(note that $L_\mathrm{min}(\mfd)$ must always be even).
This establishes the length results of Theorem~\ref{t-results-min}
in all cases but $\mfd=L(3,1)$.

\begin{table}[htb]
\[ \small \begin{array}{l|r|r|r|r|r|r|r|r||r}
    \mbox{Base level $\base{\mfd}$}
        & 2 & 3 & 4 & 5 & 6 & 7 & 8 & 9 & \mbox{Total} \\
    \hline
    \mbox{Manifolds with $L_\mathrm{min}(\mfd)=0$} &
        6 &  7 & 13 & 23 & 47 & 106 & 235 & 575 & 1\,012 \\
    \mbox{Manifolds with $L_\mathrm{min}(\mfd)=2$} &
          &    &  1 &  6 & 12 &  33 &  89 & 244 &    385 \\
    \mbox{Manifolds with $L_\mathrm{min}(\mfd)=4$} &
          &    &    &  2 & 10 &  19 &  53 & 161 &    245 \\
    \mbox{Manifolds with $L_\mathrm{min}(\mfd)=6$} &
          &  1 &    &    &  4 &  10 &  37 &  88 &    140 \\
    \mbox{Manifolds with $L_\mathrm{min}(\mfd)=8$} &
          &    &    &    &  1 &   5 &  11 &  46 &     63 \\
    \mbox{Manifolds with $L_\mathrm{min}(\mfd)=10$} &
          &    &    &    &    &   1 &   7 &  22 &     30 \\
    \mbox{Manifolds with $L_\mathrm{min}(\mfd)=12$} &
          &  1 &    &    &    &   1 &   2 &   7 &     11 \\
    \mbox{Manifolds with $L_\mathrm{min}(\mfd)=14$} &
          &    &    &    &    &     &   1 &   5 &      6 \\
    \mbox{Manifolds with $L_\mathrm{min}(\mfd)=16$} &
          &    &    &    &    &     &   1 &   4 &      5 \\
    \mbox{Manifolds with $L_\mathrm{min}(\mfd)=18$} &
          &    &    &    &    &     &     &   2 &      2 \\
    \hline
    \mbox{Manifolds with $18 < L_\mathrm{min}(\mfd) < \infty$} &
        0 &  0 &  0 &  0 &  0 &   0 &   0 &   0 & 0 \\
    \mbox{Manifolds with $L_\mathrm{min}(\mfd)=\infty$} &
        1 &  0 &  0 &  0 &  0 &   0 &   0 &   0 & 1
\end{array} \]
\caption{Bounding lengths of paths between nodes at level $\base{\mfd}$}
\label{tab-length-min}
\end{table}

For the special case $L(3,1)$,
there are precisely two one-vertex triangulations
in the census with $n=2$ tetrahedra; their isomorphism signatures are
\texttt{cMcabbgaj} and \texttt{cMcabbjak}.
A quick search shows that these can be connected using six {\mvbc} and
{\mvcb} moves:
\begin{align*} 
    \mathtt{cMcabbgaj}
   &\xrightarrow{\text{~{\mvbc}~}}
    \mathtt{dLQacccbgfg}
    \xrightarrow{\text{~{\mvbc}~}}
    \mathtt{eLPkbcdddackff}
    \xrightarrow{\text{~{\mvbc}~}}
    \mathtt{fvPQccdeedegovggo} \\
   &\xrightarrow{\text{~{\mvcb}~}}
    \mathtt{eLPkbcdddacrkk}
    \xrightarrow{\text{~{\mvcb}~}}
    \mathtt{dLQacccbgfo}
    \xrightarrow{\text{~{\mvcb}~}}
    \mathtt{cMcabbjak}
\end{align*}
See the appendix for details on how to ``decode'' these isomorphism
signatures back into full 3-manifold triangulations.
Since $H_\mathrm{min}(L(3,1))=3$, no fewer than six moves are possible.

\begin{remark}
    Table~\ref{tab-length-min} shows two ``outlier'' manifolds
    with a small base level $\base{\mfd}=3$ but
    unusually large bounds $L_\mathrm{min}(\mfd)=6$ and $12$.
    These are the lens spaces $L(4,1)$ and $L(5,2)$, which are the only
    closed prime orientable 3-manifolds for which $\base{\mfd}$ is
    \emph{not} the size of a minimal triangulation (Lemma~\ref{l-base}).

    The space $L(4,1)$ has seven one-vertex triangulations with $n=3$
    tetrahedra, and gives bounds
    $H_\mathrm{min}=2$ and $L_\mathrm{min}=12$.
    The space $L(5,2)$ has five one-vertex triangulations with $n=3$
    tetrahedra, and gives bounds
    $H_\mathrm{min}=1$ and $L_\mathrm{min}=6$.
\end{remark}


\subsection{Parallelisation and performance} \label{s-analysis-perf}

For large $n$, the algorithms that bound simplification paths
(Algorithms~\ref{a-height}, \ref{a-height-hybrid} and \ref{a-length})
all have substantial running times and memory requirements.
This is due to the large number of nodes that they process at levels
$n$ and $n+1$ (and in some cases, $n+2$)
of the corresponding Pachner graphs.

\begin{figure}[htb]
    \centering
    \subfigure[Running times]{%
        \label{sub-perf-time} \includegraphics[scale=0.5]{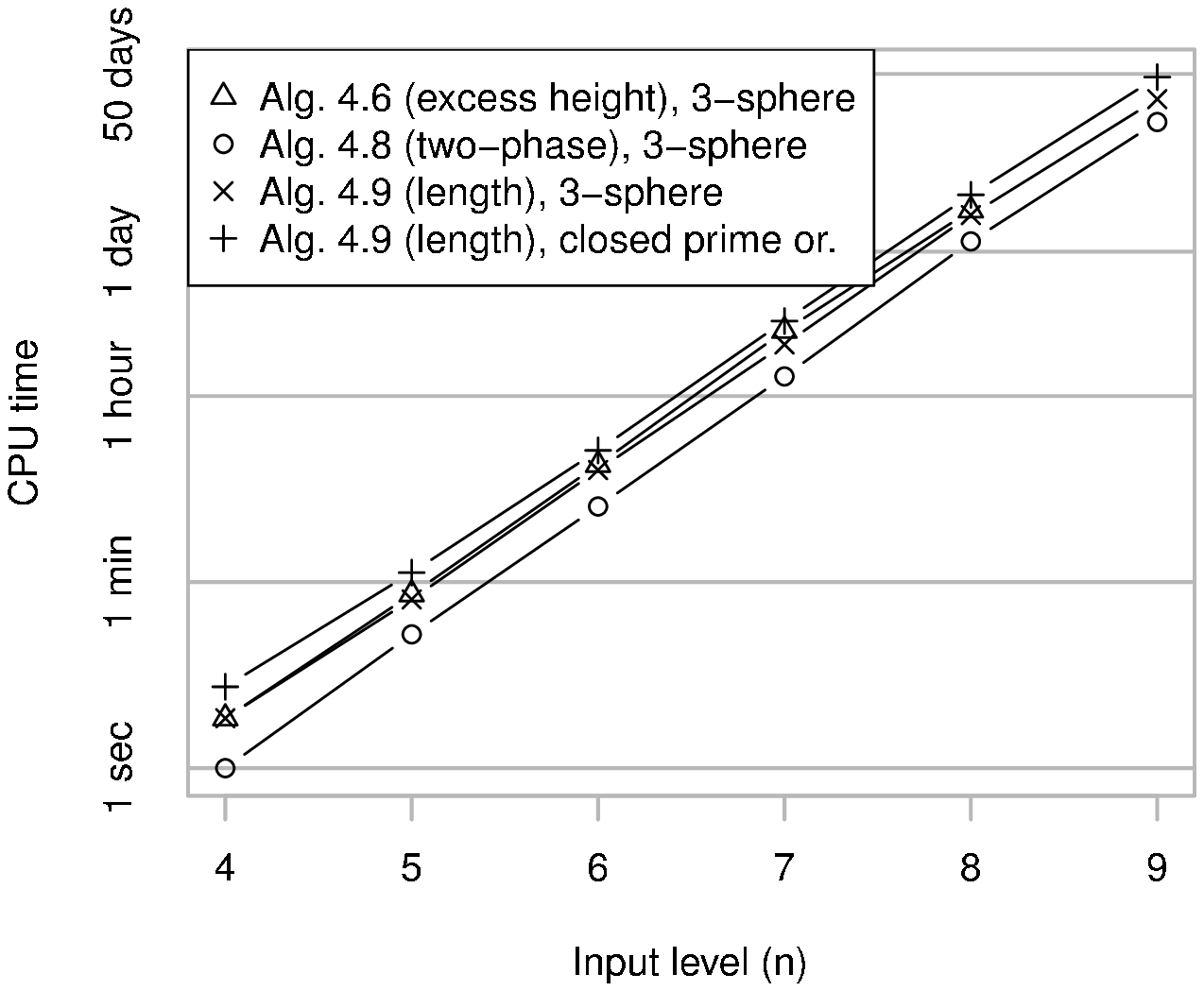}}
        \qquad
    \subfigure[Memory use]{%
        \label{sub-perf-mem} \includegraphics[scale=0.5]{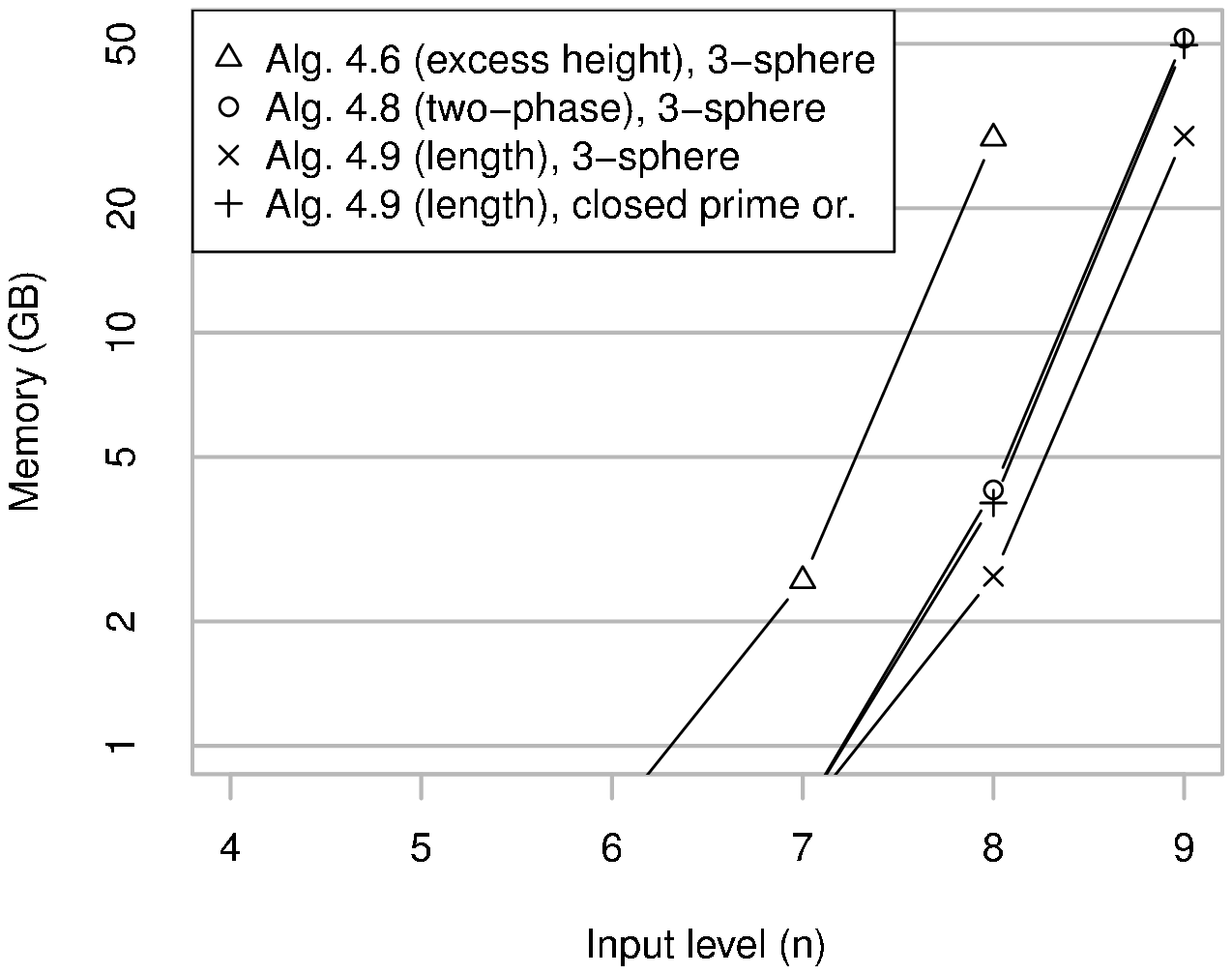}}
    \caption{Performance of algorithms from Section~\ref{s-analysis}}
    \label{fig-perf}
\end{figure}

Figure~\ref{fig-perf} plots the actual running time
and memory consumption from each of the calculations carried out in
Sections~\ref{s-analysis-height} and \ref{s-analysis-bfs}.
To recap, these calculations were:
\begin{itemize}
    \item bounding excess heights
    by running the original Algorithm~\ref{a-height} and the two-phase
    Algorithm~\ref{a-height-hybrid} over all
    3-sphere triangulations of size~$n$;
    \item bounding path lengths
    by running Algorithm~\ref{a-length} over all
    3-sphere triangulations of size~$n$, and over all size~$n$
    triangulations of closed prime orientable 3-manifolds.
\end{itemize}

Both plots use a log scale for the vertical axis.
We omit $n=3$ (where time was negligible) and do not show
memory under 1\,GB (where overheads dominate).
All computations were run on
an 8-core 2.93\,GHz Intel Xeon X5570 CPU with 72\,GB of RAM.
Algorithm~\ref{a-height} only shows $n \leq 8$, since the
case $n=9$ required too much memory to run (recall that this was
why we developed the two-phase Algorithm~\ref{a-height-hybrid} to replace it).
The data points for Algorithm~\ref{a-length} on
closed prime orientable 3-manifolds represent sums taken over all such
manifolds.\footnote{%
    This is because Algorithm~\ref{a-length} can happily process
    all such manifolds together in the same run, avoiding the messy task of
    identifying beforehand which triangulations represent which specific
    manifolds.}

We can parallelise each of these algorithms by processing nodes simultaneously:
for Algorithms~\ref{a-height} and~\ref{a-height-hybrid} we
simultaneously process nodes at the same level of the subgraph $G$,
and for Algorithm~\ref{a-length} we simultaneously process nodes
at the same distance from the source set $S$.
However, we must be careful to serialise any updates to global structures
(these include the subgraph $G$ and union-find structures in
Algorithms~\ref{a-height} and~\ref{a-height-hybrid},
and the queue of nodes to be processed in Algorithm~\ref{a-length}).
Because these global structures can grow super-exponentially large, we use
a shared memory model on a single multi-core machine,
and avoid distributed processing.

Figure~\ref{sub-perf-time} shows the total CPU time summed over all threads
of execution.  When parallelised over all eight cores,
the wall time was close to $1/8$ of these figures---for instance,
the peak 46.6 days of CPU time for Algorithm~\ref{a-length}
(closed prime orientable 3-manifolds, $n=9$)
represented just 5.9 days of wall time.
Despite the serialisation bottlenecks, all three algorithms achieved
over $98\%$ CPU utilisation for the largest cases.
This indicates that the task of
identifying and following arcs in the Pachner graphs (which could be
parallelised) was significantly more expensive than querying and updating
the global structures (which could not).

For connecting minimal triangulations,
Algorithms~\ref{a-height-min} and~\ref{a-length-min} are cheap
in comparison:
these required just 2.5 and 6.2 minutes of CPU time
respectively to process the entire census of minimal triangulations,
and each used under 1\,GB of memory.

%
%

\section{Pathological cases} \label{s-path}

In this section we first examine the most difficult triangulation to simplify
from our tables, and we show how this pathological case corresponds
to the ``exceptional'' brick $B_5$ in the Martelli-Petronio census
\cite{martelli01-or9}.  Following this, we construct larger
triangulations of graph manifolds that show how our excess height results
do not generalise for arbitrary closed prime orientable 3-manifolds.
In Section~\ref{s-conc} we return to the important case of the 3-sphere,
where no such counterexamples are known.

In the discussions below we use alphanumeric isomorphism signatures
to identify specific 3-manifold triangulations.
The appendix of this paper includes a full specification
detailing how such signatures encode tables of tetrahedron face gluings.
Alternatively, the software package {\regina} \cite{regina} can
be used to ``decode'' these signatures back into 3-manifold triangulations.

\subsection{Triangulations requiring many moves}

Recall Theorem~\ref{t-results-length}, where we show that one-vertex
triangulations of closed prime orientable 3-manifolds of size
$n \leq 9$ can be simplified using at most 17 Pachner moves.
Although this bound is tight if we allow at most two extra tetrahedra
(Lemma~\ref{l-length-tight}), the worst case is extremely rare:
just \emph{one} of the $50\,778\,377$ triangulations
in Table~\ref{tab-length-or} requires all 17 moves.
We denote this worst-case triangulation by $\tri_\mathrm{max}$.
This triangulation is an outlier (the next-worst cases require
just 13 moves), and we examine it here in detail.

$\tri_\mathrm{max}$ has
isomorphism signature \texttt{jLAMLLQbcbdeghhiixxnxxjqisj},
contains $n=9$ tetrahedra, and represents a {\sfslong} over the 2-sphere
with three exceptional fibres; specifically, $\sfs{S^2}{(2,1)\ (3,1)\ (11,-9)}$.
It has the structure of an \emph{augmented triangular
solid torus}, a common construction for {\sfslong}s
\cite{burton03-thesis,matveev98-or6}:
we build a 3-tetrahedron triangular prism, glue the triangular ends together,
and attach either a layered solid torus or a {\mobius} band
to each rectangular face.
Figure~\ref{fig-aug} outlines the construction;
see \cite{burton03-thesis} for full details and notation.


\begin{figure}[htb]
    \centering
    \includegraphics[scale=0.9]{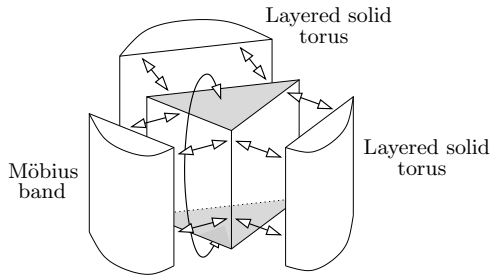}
    \caption{Building the augmented triangular solid torus $\tri_\mathrm{max}$}
    \label{fig-aug}
\end{figure}

This {\sfslong} has only one
\emph{minimal} triangulation, with $n=8$ tetrahedra and
isomorphism signature \texttt{iLLLPQcbcgffghhhtsmhgosof}.
We denote this by $\tri_\mathrm{min}$.
Unlike $\tri_\mathrm{max}$, the 8-tetrahedron $\tri_\mathrm{min}$
is constructed not from prisms and layerings, but from the brick
$B_5$ described by Martelli and Petronio in their census paper
\cite{martelli01-or9}.
This brick $B_5$ is an 8-tetrahedron triangulation of the
bounded {\sfslong} $\sfs{D}{(2,1)\ (3,1)}$, where $D$
denotes the 2-dimensional disc, and the three boundary edges
of $B_5$ describe relatively complex curves on its torus boundary.

The brick $B_5$ is unique in the Martelli-Petronio census
in that it is the only large brick that allows {\sfslong}s
to be built using fewer tetrahedra than standard
prism-and-layering constructions.
However, it only appears rarely in the census; this is due
to the specific parameters $\sfs{D}{(2,1)\ (3,1)}$,
the complex boundary curves, and the large number of tetrahedra.
Indeed, amongst all minimal triangulations of
closed prime orientable 3-manifolds of size $\leq 8$,
the triangulation $\tri_\mathrm{min}$ is the \emph{only} one
to contain $B_5$ at all.

In a sense then, the 17 moves needed to simplify $\tri_\mathrm{max}$
indicate that $B_5$ is ``substantially different'' from standard
prism-and-layering constructions: because $\tri_\mathrm{min}$ is the
only smaller triangulation of the same manifold, we are forced to
reorganise the tetrahedra of $\tri_\mathrm{max}$ to form a $B_5$
configuration in order to simplify it.

It is tempting to use $B_5$ to search for larger triangulations that
require even more moves to simplify, but initial attempts are unsuccessful.
Stepping up to nine tetrahedra, there are
three minimal triangulations of closed prime orientable
3-manifolds that include the brick $B_5$: these represent the spaces
$\sfs{S^2}{(2,1)\ (3,1)\ (13,-11)}$,
$\sfs{S^2}{(2,1)\ (3,1)\ (16,-13)}$ and
$\sfs{S^2}{(2,1)\ (3,1)\ (17,-14)}$.
The corresponding augmented triangular solid tori each have $n=10$ tetrahedra,
and a computer search shows that each can be simplified
using 17~Pachner moves once again.

\subsection{Triangulations requiring greater excess height}

In Theorems~\ref{t-results-height} and \ref{t-results-min}, we show that
one-vertex triangulations of closed prime orientable 3-manifolds of size
$n \leq 9$ can be simplified using at most two extra tetrahedra,
and that any two minimal triangulations of such a manifold can be
connected using at most two extra tetrahedra.  Here we show that for
arbitrary manifolds, these results do not generalise: in particular,
we construct triangulations of graph manifolds with $n=10$ tetrahedra
for which three extra tetrahedra are required.

All of the graph manifolds described in this section are obtained by
joining two bounded {\sfslong}s along their torus boundaries according to
a given $2 \times 2$ matching matrix.
We refer the reader to \cite{burton07-nor10} for further details and
notation.

\begin{theorem} \label{t-bad-simplify}
    There exists a closed prime orientable 3-manifold $\mfd$
    and a node $\nu \in \rpg{\mfd}$ at level $>\base{\mfd}$
    from which every simplification path has excess height $\geq 3$.
\end{theorem}

\begin{proof}
    Let $\mfd$ be the graph manifold
    $\graphmfd{\sfs{D}{(2,1)\ (3,1)}}{\sfs{D}{(3,2)\ (3,2)}}
        {\homtwotable{-1}{1}{1}{0}}$.
    The census of minimal triangulations in \cite{burton07-nor10} 
    shows that $\mfd$ has just one minimal triangulation; this
    has $\base{\mfd}=9$ tetrahedra and isomorphism signature
    \texttt{jLLALPQaceefgihhijkuxpwhwns}.
    
    Let $\nu$ be the node of $\rpg{\mfd}$ that describes the
    triangulation with size $n=10$ and isomorphism signature
    \texttt{kLLzLQAkaceiggghijjjkxuaatlsqw}.
    A computer search shows that no path from $\nu$ with excess height
    $\leq 2$ can reach level~9 of $\rpg{\mfd}$.  However, with excess
    height 3 we obtain a simplification path that connects $\nu$ with
    the unique minimal triangulation described above.
\end{proof}

The example above was found as a side-effect of processing the
10-tetrahedron census of minimal triangulations \cite{burton07-nor10}:
it was difficult to identify which manifold the triangulation
represented, and this was only resolved after (with some difficulty)
converting it into a known triangulation of the graph manifold $\mfd$.
No other pathological cases have been found in this way.
The 10-tetrahedron triangulation above is not based on a standard
prism-and-layering construction, although the 9-tetrahedron minimal
triangulation of $\mfd$ is.

\begin{theorem} \label{t-bad-join}
    There exists a closed prime orientable 3-manifold $\mfd$
    and nodes $\nu_1,\nu_2 \in \rpg{\mfd}$ at level $\base{\mfd}$
    for which every path that joins $\nu_1$ with $\nu_2$
    has excess height $\geq 3$.
\end{theorem}

\begin{proof}
    Recall that, unlike the census of \emph{all} triangulations of
    closed prime orientable 3-manifolds, the
    census of all \emph{minimal} triangulations of such manifolds extends
    to 11 tetrahedra \cite{burton11-genus}.  By running
    Algorithm~\ref{a-height-min} over this extended census,
    we identify seven manifolds
    $\mfd$ for which $H_\mathrm{min}(\mfd)=3$ (and none for which
    $H_\mathrm{min}(\mfd) > 3$).  By Lemma~\ref{l-height-min-tight},
    it follows that each of the corresponding Pachner graphs
    $\rpg{\mfd}$ has nodes $\nu_1,\nu_2$ with the property above.
    Table~\ref{tab-bad-join} summarises these results.
\end{proof}

\begin{table}[htb]
    \small
    \centering
    \begin{tabular}{c|l|c|c}
    {Level} &
    \multicolumn{1}{c|}{Manifold $\mfd$} &
    {Nodes at} &
    {Bound} \\
    $\base{\mfd}$ & &
    {level $\base{\mfd}$} &
    $H_\mathrm{min}(\mfd)$ \\
    \hline
    10 & $\graphmfd{\sfs{D}{(2,1)\ (3,1)}}
                   {\sfs{D}{(3,2)\ (3,2)}}
                   {\homtwotable{1}{-1}{0}{1}}$ &
         6 & 3 \\
    \hline
    11 & $\graphmfd{\sfs{D}{(2,1)\ (3,1)}}
                   {\sfs{D}{(3,2)\ (3,2)}}
                   {\homtwotable{-1}{1}{1}{-2}}$ &
         18 & 3 \\
    11 & $\graphmfd{\sfs{D}{(2,1)\ (3,1)}}
                   {\sfs{D}{(3,2)\ (4,1)}}
                   {\homtwotable{1}{-1}{0}{1}}$ &
         9 & 3 \\
    11 & $\graphmfd{\sfs{D}{(2,1)\ (3,1)}}
                   {\sfs{D}{(3,2)\ (4,3)}}
                   {\homtwotable{1}{-1}{0}{1}}$ &
         9 & 3 \\
    11 & $\graphmfd{\sfs{D}{(2,1)\ (3,1)}}
                   {\sfs{D}{(3,2)\ (5,2)}}
                   {\homtwotable{1}{-1}{0}{1}}$ &
         9 & 3 \\
    11 & $\graphmfd{\sfs{D}{(2,1)\ (3,1)}}
                   {\sfs{D}{(3,2)\ (5,3)}}
                   {\homtwotable{1}{-1}{0}{1}}$ &
         9 & 3 \\
    11 & $\graphmfd{\sfs{D}{(2,1)\ (4,1)}}
                   {\sfs{D}{(3,2)\ (3,2)}}
                   {\homtwotable{1}{-1}{0}{1}}$ &
         6 & 3 \\
    \end{tabular}
    \caption{Manifolds from the 11-tetrahedron census with
        $H_\mathrm{min}(\mfd)>2$}
    \label{tab-bad-join}
\end{table}

As with Theorem~\ref{t-bad-simplify}, all of the manifolds in
Table~\ref{tab-bad-join} are graph manifolds.  We return to this issue
in Section~\ref{s-conc},
where we discuss the lack of known pathological examples for
``simple'' manifolds (in particular, the 3-sphere).

%
%

\section{Discussion} \label{s-conc}

As noted already, the bounds on simplification paths that
we obtain in Section~\ref{s-analysis} are astonishingly small.
Although we only consider triangulations of size $n \leq 9$,
this is not a small sample:
the census includes $\sim 150$ million triangulations,
including $\sim 31$~million one-vertex 3-spheres,
plus another
$\sim 51$~million one-vertex triangulations of $1\,900$ distinct closed prime
orientable 3-manifolds that cover all eight Thurston geometries as well as
non-geometric combinations of these \cite{matveev03-algms}.

These results have interesting implications for
simplification algorithms.  Given a triangulation $\tri$
(of arbitrary size) that we wish to
simplify, we can follow the strategy of Algorithm~\ref{a-length}
and perform a breadth-first search from $\tri$ using Pachner moves and
octahedron, pillow and prism flips, never adding more than two extra
tetrahedra.
Theorem~\ref{t-results-height} gives us hope that this will
indeed simplify $\tri$,
and Theorem~\ref{t-results-length} gives us hope that the search will be
relatively fast.
Initial observations from using such simplification techniques
in ongoing research projects show them to be remarkably successful.

Triangulations of the 3-sphere exhibit particularly good behaviour.
In contrast to closed prime orientable manifolds, the worst-case
number of moves needed to simplify a 3-sphere triangulation
barely grows at all with the size $n$ (Figure~\ref{fig-steps}).
Moreover, unlike the graph manifolds discussed in Section~\ref{s-path},
there are no known pathological 3-sphere triangulations that require
more than two extra tetrahedra to simplify, despite the 3-sphere having
far more triangulations than any other manifold in the census.
These results lead us to make the following conjecture:

\begin{conjecture} \label{cj-boundedheight}
    From any node at any level $n \geq 3$ of the graph $\rpg{\sss}$,
    there is a simplification path of excess height $\leq 2$.
\end{conjecture}

If true, this would help explain why 3-sphere
triangulations are so easy to simplify in practice.
Moreover, by combining this result with Theorem~\ref{t-numvert},
we could reduce Mijatovi{\'c}'s bound on the number of Pachner moves
needed to simplify a one-vertex triangulation of the 3-sphere
from $\exp(O(n^2))$ down to $\exp(O(n\log n))$ instead.\footnote{%
    Even if the worst-case excess height for simplification paths
    in $\rpg{\sss}$ grows at a rate of $O(n)$, we can still reduce
    Mijatovi{\'c}'s bound to $\exp(O(n\log n))$ in this way.}

There are theoretical reasons to believe that
Conjecture~\ref{cj-boundedheight} might be true.
Similar results are known for the simplification of knots in $\R^3$:
Dynnikov shows that an arc presentation of the trivial knot
can be simplified using elementary moves without increasing the
complexity of the diagram \cite{dynnikov03-knot,dynnikov06-arc},
and Henrich and Kauffman use his results to show that any
projection of the unknot can be simplified using Reidemeister moves with only
a quadratic increase in the number of crossings \cite{henrich11-quadratic}.
In the setting of triangulation simplification,
Mijatovi{\'c}'s bounds for the 3-sphere---though
extremely large---remain far smaller than for other manifolds
(whose bounds involve towers of exponentials).
Moreover, every pathological triangulation from Section~\ref{s-path}
that requires three extra tetrahedra is of a non-geometric graph manifold:
it is possible that the boundaries between geometric components in such
manifolds act as barriers to simplification
that would not exist for the 3-sphere.


We now turn our attention to path length; that is, the number of Pachner moves
required to simplify a triangulation.
Our next conjecture involves the \emph{average} simplification path lengths
$\means{n}$ (for 3-sphere triangulations) and
$\meanor{n}$ (for closed prime orientable 3-manifold triangulations),
as defined in Section~\ref{s-analysis-bfs}.

\begin{conjecture}
    Both quantities $\means{n}$ and $\meanor{n}$
    are in $O(1)$; that is, they are bounded above by a constant
    that does not depend on $n$.
\end{conjecture}

This conjecture is clearly supported by the results in
Figure~\ref{fig-steps}.  Looking more closely at
Tables~\ref{tab-length-s3} and~\ref{tab-length-or} however,
we can formulate a more interesting observation:

\begin{conjecture} \label{cj-generic}
    For a given 3-manifold $\mfd$, let $\phi_n(\mfd)$ denote the fraction
    of nodes at level $n$ of $\rpg{\mfd}$ that have an arc leading
    directly down to level $n-1$.
    Then, if $\mfd$ is the 3-sphere or a closed prime orientable
    3-manifold, $\lim_{n \to \infty} \phi_n(\mfd) = 1$.
\end{conjecture}

In other words, as $n \to \infty$, almost all size~$n$ triangulations of
the 3-sphere or of closed prime orientable 3-manifolds can be
simplified immediately, using just a single {\mvcb} move.  Intuitively this
seems reasonable, since with more tetrahedra it is more likely that a
suitable degree~three edge can be found.  However, proving it appears
difficult: like most probabilistic arguments involving 3-manifold
triangulations, we run into the critical problem that almost all
pairwise gluings of tetrahedron faces do not represent 3-manifolds
at all \cite{dunfield06-random-covers}.

This observation can be framed in terms of
\emph{generic complexity} \cite{myasnikov08-crypto},
where we are allowed to ignore a vanishingly small population of pathological
inputs.  In this context, Conjecture~\ref{cj-generic} states
that generic triangulations of $\sss$ or of closed prime orientable
3-manifolds can be simplified in a single move.  An interesting direction
for future research is to build a 3-sphere recognition algorithm based on
Pachner moves that runs in polynomial time for generic inputs.

We return now to \emph{worst case} path length for 3-sphere triangulations.
Again, we observe from Figure~\ref{fig-steps} that the worst-case bound
$L_n(\sss)$ barely changes at all for $3 \leq n \leq 9$: there is just
one increment, which is of the smallest possible size (recalling that $L_n$
must always be odd).  This is in contrast to the case of arbitrary closed
prime orientable manifolds, where the worst-case path length grows with $n$,
and does so at a clearly accelerating rate.

This slow rate of growth for $L_n(\sss)$ has direct implications for
the complexity of 3-sphere recognition:

\begin{observation}
    If $L_n(\sss) \in o(n/\log n)$, then there is a 3-sphere recognition
    algorithm with running time in $o(\alpha^n)$ for any $\alpha>1$.
    That is, the 3-sphere can be recognised in sub-exponential time.
\end{observation}

\begin{proof}
    From any $n$-tetrahedron triangulation $\tri$ there are at most $O(n)$
    possible {\mvbc} or {\mvcb} moves (one for each face or edge of $\tri$).
    We can therefore enumerate all possible sequences of $k$ moves in
    $O\left(n^t \cdot (n+k)^k\right)$ time,
    where $n^t$ is a small polynomial term that accounts for the mechanics
    of performing each {\mvbc} or {\mvcb} move.
    Setting $k=L_n(\sss)$ and observing that $n+k \in O(n)$,
    we find that if $\tri$ can be simplified, this
    can be done in $O\left(n^{t+L_n(\sss)}\right)$ time.

    Repeating this operation $n-2$ times, we can either simplify $\tri$
    to a known 2-tetrahedron triangulation of $\sss$, or else show
    this to be impossible.  The total running time is
    \[ O\left(n \cdot n^{t+L_n(\sss)}\right)
    = O\left(n^{t+1} \cdot \exp\left(\log n \cdot L_n(\sss)\right)\right).\]
    Since $L_n(\sss) \in o(n/\log n)$, this running time
    is $o(\alpha^n)$ for any $\alpha > 1$.
\end{proof}

Although there is too little data to predict the precise growth rate
of $L_n(\sss)$ (a situation that Theorem~\ref{t-numvert} shows is
extremely difficult to rectify), it is certainly plausible to suggest
from Figure~\ref{fig-steps} that $L_n(\sss)$ is ``sufficiently sublinear''
to satisfy $L_n(\sss) \in o(n/\log n)$.
This leads us to the following conjecture:

\begin{conjecture} \label{cj-boundedmoves}
    From any node at any level $n \geq 3$ of the graph $\rpg{\sss}$,
    there is a simplification path of length $\leq L_n(\sss)$
    where $L_n(\sss) \in o(n/\log n)$, and therefore
    sub-exponential time 3-sphere recognition is possible.
\end{conjecture}

If this conjecture could be proven, the resulting
sub-exponential time 3-sphere recognition algorithm would be a significant
breakthrough in algorithmic 3-manifold topology.

%
%

\section*{Acknowledgements}

The author is grateful to the Australian Research Council for their
support under the Discovery Projects funding scheme (projects
DP1094516 and DP110101104).
Computational resources used in this work
were provided by the Queensland Cyber Infrastructure Foundation
and the Victorian Partnership for Advanced Computing.

%
%

\appendix
\section*{Appendix: Isomorphism signatures}

In Section~\ref{s-isosig} we describe isomorphism signatures, which are
small pieces of data that uniquely define the isomorphism class of a
triangulation, and that can be computed in small polynomial time.
Recall that to compute an isomorphism signature, we must:
\begin{enumerate}[(i)]
    \item enumerate all $O(n)$ canonical labellings of the
    input triangulation;
    \item encode the full set of face gluings for each canonical labelling;
    \item choose the lexicographically smallest encoding amongst all
    canonical labellings.
\end{enumerate}

In this appendix we specify the encoding used in step~(ii).
For the theoretical arguments in Section~\ref{s-isosig},
it is sufficient to use a simple bit
sequence that encodes the complete table of face gluings.
In practice however, we use a more compact method:
\begin{itemize}
    \item we do not encode redundant information
    from the face gluings table;
    \item we use \emph{printable characters} (letters, digits and
    punctuation) instead of plain bits.
\end{itemize}
In this way, the encoding is kept short but remains easy to write or
type by hand.  This is useful in both papers and software:
authors can describe triangulations precisely without presenting
large tables of face gluings (as seen in Section~\ref{s-path} of this paper),
and readers can reconstruct these triangulations quickly using the
appropriate software.

The precise encoding described here is implemented in version~4.90 of
the freely-available software package {\regina} \cite{regina}.
This encoding method is defined not only for closed triangulations (as
described in this paper), but also for triangulations with boundary
(where some tetrahedron faces are not glued to any partner face) and
ideal triangulations (where some vertex neighbourhoods are not bounded
by spheres).
The encoding draws upon and extends ideas from the dehydration format of
Callahan et~al.\ \cite{callahan99-cuspedcensus}.

\subsection*{Encoding integers}

A key piece of the encoding process is to translate integers into
printable characters.  We describe two methods: one for ``small'' integers in
the range $0,\ldots,63$, and one for ``large'' integers in the range
$0,\ldots,n$.

For small integers, we use the following translation table from $0,\ldots,63$
to printable characters.  This is similar to, but not the same as,
a typical Base64 encoding table.\footnote{A key difference is that we
use lower-case letters first, to simplify the process of manually typing
isomorphism signatures on the keyboard.}
For each $i \in \{0,\ldots,63\}$ we let $\pi(i)$ denote the
corresponding printable character.

\[\begin{array}{l|ccc|ccc|ccc|c|c}
    \mbox{Integer $i$} & 0 & \cdots & 25 & 26 & \cdots & 51 & 52 &
        \cdots & 61 & 62 & 63 \\
    \hline
    \mbox{Printable character $\pi(i)$} &
        \mathtt{a} & \cdots & \mathtt{z} &
        \mathtt{A} & \cdots & \mathtt{Z} &
        \mathtt{0} & \cdots & \mathtt{9} &
        \mathtt{+} & \mathtt{-}
\end{array}\]

To encode large integers we use $d$ printable characters,
where $d=\lfloor \log_{64}(n) \rfloor + 1$.
Specifically: for any integer $i \in \{0,\ldots,n\}$, we write
$i$ as a $d$-digit number in base~64, and then use $\pi(\cdot)$ to
encode the individual digits as described above, beginning with the
least significant digit.  We denote the resulting string of
characters by $\epsilon(i)$.

For example, suppose that $n=100$ and we wish to encode the integer $i=93$.
This encoding uses $d=\lfloor \log_{64}(100)\rfloor+1 = 2$ characters.
Since $93=29+1\cdot64$, the two base~64 ``digits'' of $i$ are $(29, 1)$.
The final encoding is then $\epsilon(93)=\pi(29)\pi(1)=\mathtt{Db}$.

If $i \ll n$, we may need to pad the base~64 representation of $i$ with leading
zeroes.  For instance, if $n=100$ and we wish to encode $i=5$, the
two base~64 ``digits'' are $(5,0)$, and the final encoding is
$\epsilon(5)=\pi(5)\pi(0)=\mathtt{fa}$.

\subsection*{Encoding face gluings}

Here we identify the information required to reconstruct the
full table of face gluings for a canonical labelling.
Such a table can be seen in Table~\ref{tab-gluings}
(from Section~\ref{s-isosig}), which we use as an example
throughout this discussion.

We begin by removing redundant information from the gluings table.
Suppose that face $f$ of tetrahedron $t$ is glued to face $f'$ of tetrahedron
$t'$.  If $(t',f') < (t,f)$ lexicographically\footnote{%
    That is, either $t'<t$, or else $t'=t$ and $f'<f$.},
then we delete the table cell in position $(t,f)$.
This is because the same gluing has already
been seen from the other direction in cell $(t',f')$, and so the
information in cell $(t,f)$ is redundant.
Table~\ref{tab-gluings-redundant} shows a copy of
Table~\ref{tab-gluings} with these redundant cells crossed out.

\begin{table}[htb]
    \newcommand{\gap}{\hspace{2ex}}
    \newcommand{\bglue}[1]{\cancel{#1}}
    \centering
    \small
    \begin{tabular}{l|c|c|c|c}
    &
    \multicolumn{1}{c|}{Face 0} &
    \multicolumn{1}{c|}{Face 1} &
    \multicolumn{1}{c|}{Face 2} &
    \multicolumn{1}{c}{Face 3} \\
    & Vertices 123 & Vertices 023 & Vertices 013 & Vertices 012 \\
    \hline
    {Tet.\ 0} & Tet.\ 0:\gap120 & Tet.\ 1:\gap023 &
                Tet.\ 2:\gap013 & \bglue{Tet.\ 0:\gap312} \\
    {Tet.\ 1} & Tet.\ 2:\gap230 & \bglue{Tet.\ 0:\gap023} &
                Tet.\ 1:\gap012 & \bglue{Tet.\ 1:\gap013} \\
    {Tet.\ 2} & Tet.\ 2:\gap012 & \bglue{Tet.\ 1:\gap312} &
                \bglue{Tet.\ 0:\gap013} & \bglue{Tet.\ 2:\gap123}
    \end{tabular}
    \caption{Removing redundant information from the face gluings table}
    \label{tab-gluings-redundant}
\end{table}

To reconstruct the table cells that remain, we extract the following
information:
\begin{itemize}
    \item \emph{Destination sequence:}
    As in Section~\ref{s-isosig}, let $A_{t,f}$ denote the tetrahedron
    glued to face $f$ of tetrahedron $t$, so that
    $A_{t,f} \in \{\bdry, 0,\ldots,n-1\}$ for all $t=0,\ldots,n-1$ and
    $f=0,\ldots,3$.  Here we use the special symbol $A_{t,f}=\bdry$
    to indicate that face $f$ of tetrahedron $t$ is a \emph{boundary face},
    i.e., it is not glued to any partner tetrahedron face at all.
    The \emph{destination sequence} is obtained by writing out
    $A_{0,0},A_{0,1},A_{0,2},A_{0,3},\allowbreak A_{1,0},\ldots,A_{n-1,3}$,
    and then deleting any entries $A_{t,f}$ that correspond to deleted table
    cells $(t,f)$.

    For our example in Table~\ref{tab-gluings-redundant},
    the corresponding destination sequence is $0,1,2,2,1,2$.

    \item \emph{Type sequence:}
    For each term $A_{t,f}$ in the destination sequence,
    we classify the corresponding tetrahedron face as one of the
    following three types:
    \begin{itemize}
        \item \emph{Type 0}: This indicates that the face is a boundary face,
        i.e., $A_{t,f}=\bdry$.

        \item \emph{Type 1}: This indicates that the face is joined to
        a \emph{new tetrahedron}.  That is, $A_{t,f}=k$ where
        (i)~$k \geq 1$, and (ii)~the term $A_{t,f}$ is the
        \emph{first time} that $k$ appears in the destination sequence.

        \item \emph{Type 2}: This indicates that the face is joined to
        an \emph{old tetrahedron}.  That is, either $A_{t,f}=0$, or else
        $A_{t,f}=k \geq 1$ but there is some other term $A_{t',f'}=k$ that
        appears in the destination sequence earlier than $A_{t,f}$.
    \end{itemize}

    We obtain the \emph{type sequence} by writing out the
    corresponding face type for each term in the destination sequence.

    For our example, the type sequence is $2,1,1,2,2,2$.
    The two `1's in this sequence correspond to the first appearances of
    tetrahedra~1 and~2, in cells $(0,1)$ and $(0,2)$ respectively
    of the face gluings table.

    \item \emph{Permutation sequence:}
    For each non-boundary tetrahedron face $(t,f)$ that is glued to some
    partner face $(t',f')$, we identify the corresponding
    \emph{gluing permutation} $p_{t,f} \in S_4$.
    This is the permutation for which $p_{t,f}(f)=f'$, and
    each vertex $v \neq f$ of tetrahedron $t$ maps to
    vertex $p_{t,f}(v) \neq f'$ of tetrahedron $t'$ under this face gluing.

    We can represent each permutation in $S_4$ by a number:
    simply order them lexicographically and number them $0,\ldots,23$.
    The identity $(0,1,2,3) \mapsto (0,1,2,3)$ becomes 0,
    the permutation $(0,1,2,3) \mapsto (0,1,3,2)$ becomes 1, and so on
    until $(0,1,2,3) \mapsto (3,2,1,0)$ becomes 23.

    To build the \emph{permutation sequence},
    we write out the number for each permutation $p_{t,f}$ corresponding
    to each term $A_{t,f} \neq \bdry$ in the destination sequence.

    In our example, this sequence is $21,0,0,9,1,18$.
    The two `0's come from the faces of type~1: because the
    labelling is canonical, those permutations must be the identity.
\end{itemize}

\subsection*{The full isomorphism signature}

To completely encode a labelled triangulation, we bundle all of the
above information into a sequence of printable characters as follows.
\begin{itemize}
    \item We begin the sequence by encoding $n$ and $d$ (recall that
    $d$ indicates how many printable characters are required
    to encode each ``large integer'' in the range $0,\ldots,n$).
    Specifically:
    if $n \geq 63$ then we begin with the special marker
    $\pi(63)$ followed by $\pi(d)$ and $\epsilon(n)$.
    If $n < 63$, we simply begin with $\pi(n)$ and it is implicitly
    understood that $d=1$.


    Strictly speaking, because we encode $\pi(d)$, this scheme only works if
    $d < 64$.  However, this is true for all
    $n < 64^{63} \simeq 6 \times 10^{113}$, which covers all conceivable
    practical requirements.

    \item We follow this with the type sequence.
    Since each face type is in $\{0,1,2\}$, we can encode
    \emph{three terms} of the sequence in each printable character.
    Specifically, if the type sequence is $t_0,t_1,\ldots$,
    then we write $\pi(t_0+4t_1+16t_2)$, $\pi(t_3+4t_4+16t_5)$, \ldots,
    padding the sequence with one or two trailing zeroes if necessary.

    \item Next we write the destination sequence, but
    \emph{only for faces of type~2}.
    Each destination $A_{t,f}$ is encoded as $\epsilon(A_{t,f})$.
    We ignore the other faces because type~0 faces are boundary,
    and type~1 destinations can be deduced from the fact that the
    labelling is canonical.

    \item We finish with the permutation sequence, again
    only for faces of type~2.  Each permutation number $i$ is encoded
    as $\pi(i)$.
    We ignore type~1 faces because, in a canonical labelling,
    all type~1 gluings must use the identity permutation.
\end{itemize}

For our example: since $n = 3 < 63$, we begin the sequence with
$\pi(3)=\mathtt{d}$.
The type sequence $2,1,1,2,2,2$ encodes to
$\pi(2 + 1 \cdot 4 + 1 \cdot 16)$ followed by
$\pi(2 + 2 \cdot 4 + 2 \cdot 16)$; that is, $\pi(22)\pi(42)=\mathtt{wQ}$.
The destination sequence for type~2 faces is
$0,2,1,2$, which encodes to $\mathtt{acbc}$.
The permutation sequence for type~2 faces is
$21,9,1,18$, which encodes to $\mathtt{vjbs}$.
We combine these pieces to obtain the final sequence \texttt{dwQacbcvjbs}.

To obtain an isomorphism signature, we must choose the lexicographically
smallest encoding amongst all canonical labellings.
For this we order printable characters according to the standard ASCII
tables (so, for instance, $\mathtt{1} < \mathtt{Z} < \mathtt{a}$).
This keeps things simple for programmers, who can using ready-made
comparison functions such as \textit{strcmp} in C.

In our example, the encoding \texttt{dwQacbcvjbs} from
Table~\ref{tab-gluings-redundant} is not an isomorphism
signature.  Several different canonical labellings give lexicographically
smaller encodings; the smallest is
\texttt{dLQabccbcjj}, which becomes the isomorphism signature for
this triangulation.

As a final remark, we can easily extend isomorphism signatures to support
disconnected triangulations (which are not considered in this paper):
simply list the signatures for each connected component in sorted order.

%
%

{
\small
\bibliographystyle{amsplain}
\bibliography{pure}
}

%
%

\bigskip
\bigskip
\noindent
Benjamin A.~Burton \\
School of Mathematics and Physics, The University of Queensland \\
Brisbane QLD 4072, Australia \\
(bab@maths.uq.edu.au)

\end{document}